\documentclass{siamart220329}

\usepackage{color}
\usepackage{booktabs}
\usepackage{tablefootnote}
\usepackage{amssymb,amsmath}
\usepackage[noend]{algpseudocode}
\usepackage[nocompress]{cite}
\usepackage{lmodern}
\usepackage[T1]{fontenc}
\usepackage[latin9]{inputenc}

\renewcommand{\tabcolsep}{2.5pt}
\allowdisplaybreaks

\definecolor{green}{rgb}{0.0,0.5,0.0} 

\makeatletter
\def\ALG@special@indent{%
    \ifdim\ALG@thistlm=0pt\relax
        \hskip-\leftmargin
    \else
        \hskip\ALG@thistlm
    \fi
}
\newcommand{\StateO}[1]{\item[]\noindent\ALG@special@indent #1}
\algnewcommand{\LineComment}[1]{\Statex \hskip\ALG@thistlm #1}
\usepackage{tabularx}
\newcommand{\multiline}[1]{%
  \begin{tabularx}{\dimexpr\linewidth-\ALG@thistlm}[t]{@{}X@{}}
    #1
  \end{tabularx}
}
\makeatother

\theoremstyle{plain}
\newtheorem{thm}{\protect\theoremname}[section]
\newtheorem{prop}[thm]{\protect\propositionname}
\newtheorem{lem}[thm]{\protect\lemmaname}
\providecommand{\lemmaname}{Lemma}
\providecommand{\propositionname}{Proposition}
\providecommand{\theoremname}{Theorem}

\begin{document}
\title{Global Complexity Bound of a Proximal ADMM for Linearly-Constrained
Nonseparable Nonconvex Composite Programming\thanks{\textbf{Funding}: The first author has been supported by (i) the US Department of Energy
(DOE) and UT-Battelle, LLC, under contract DE-AC05-00OR22725, (ii)
the Exascale Computing Project (17-SC-20-SC), a collaborative effort
of the U.S. Department of Energy Office of Science and the National
Nuclear Security Administration, and (iii) the IDEaS-TRIAD Fellowship
(NSF Grant CCF-1740776). The second author was partially supported by ONR Grant N00014-18-1-2077 and AFOSR Grant FA9550-22-1-0088. \\
\textbf{Versions}: v0.1 (Oct. 24, 2021), v0.2 (Dec. 9, 2021), v1.0 (Jun. 14, 2022), v2.0 (Jan. 3, 2023)
}
}
\author{Weiwei Kong\thanks{Computer Science and Mathematics Division, Oak Ridge National Laboratory,
Oak Ridge, TN, 37830. \protect\protect\href{mailto:wwkong92@gmail.com}{wwkong92@gmail.com}}\hspace*{0.5em} \and Renato D.C. Monteiro\thanks{School of Industrial and Systems Engineering, Georgia Institute of
Technology, Atlanta, GA, 30332-0205. \protect\protect\href{mailto:monteiro@isye.gatech.edu}{monteiro@isye.gatech.edu}}}
\maketitle
\begin{abstract}
This paper proposes and analyzes a dampened proximal alternating direction
method of multipliers (DP.ADMM) for solving linearly-constrained nonconvex
optimization problems where the smooth part of the objective function
is nonseparable. Each iteration of DP.ADMM consists of: (ii) a sequence
of partial proximal augmented Lagrangian (AL) updates, (ii) an under-relaxed
Lagrange multiplier update, and (iii) a novel test to check whether
the penalty parameter of the AL function should be updated. Under
a basic Slater point condition and some requirements on the dampening
factor and under-relaxation parameter, it is shown that DP.ADMM obtains
an approximate first-order stationary point of the constrained problem in ${\cal O}(\varepsilon^{-3})$
iterations for a given numerical tolerance $\varepsilon>0$. One of
the main novelties of the paper is that convergence of the method
is obtained without requiring any rank assumptions on the constraint
matrices. 
\end{abstract}



\global\long\def\tx{\tilde{x}}%
\global\long\def\rn{\mathbb{R}^{n}}%
\global\long\def\R{\mathbb{R}}%
\global\long\def\r{\mathbb{R}}%
\global\long\def\n{\mathbb{N}}%
\global\long\def\c{\mathbb{C}}%
\global\long\def\pt{\mathbb{\partial}}%
\global\long\def\lam{\lambda}%
\global\long\def\argmin{\operatorname*{argmin}}%
\global\long\def\Argmin{\operatorname*{Argmin}}%
\global\long\def\argmax{\operatorname*{argmax}}%
\global\long\def\dom{\operatorname*{dom}}%
\global\long\def\ri{\operatorname*{ri}}%
\global\long\def\inner#1#2{\langle#1,#2\rangle}%
\global\long\def\trc{\operatorname*{tr}}%
\global\long\def\cConv{\overline{{\rm Conv}}\ }%
\global\long\def\intr{\operatorname*{int}}%

\section{Introduction}
\label{sec:intro}

Consider the following composite optimization
problem:
\begin{equation}
\min_{x\in\rn}\left\{ \phi(x):=f(x)+h(x):Ax=d\right\}, \label{eq:main_prb}
\end{equation}
where $h$ is a closed convex function, $f$ is a (possibly) nonconvex differentiable
function on the domain of $h$, the gradient of $f$ is Lipschitz continuous, $A$ is a linear operator, $d\in \r^{\ell}$ is a vector in the image of $A$ (denoted as ${\rm{Im}}(A)$), and
the following $B$-block structure is assumed:
\begin{gather}
\begin{gathered}
n=n_{1}+\ldots+n_{B},\quad x=(x_{1},\ldots,x_{B})\in\r^{n_{1}}\times\cdots\times\r^{n_{B}}\\
h(x)=\sum_{t=1}^{B}h_{t}(x_{t}),\quad Ax=\sum_{t=1}^{B}A_{t}x_{t},
\end{gathered}\label{eq:block_structure}
\end{gather}
where $\{A_{t}\}_{t=1}^{ B}$ is another set of linear operators and $\{h_{t}\}_{t=1}^{ B}$
is another set of proper closed convex functions with compact domains. 

Due to the block structure in \eqref{eq:block_structure}, a popular algorithm
for obtaining stationary points of \eqref{eq:main_prb} is the proximal alternating direction
method of multipliers (ADMM) wherein a sequence of smaller augmented Lagrangian
type subproblems is solved over $x_1,...,x_B$ sequentially or in parallel. 
However, the main drawbacks of existing ADMM-type methods include: (i) strong
assumptions about the structure of $h$; (ii) iteration complexity bounds that scale
poorly with the numerical tolerance; (iii) small stepsize parameters; 
or (iv) a strong rank assumption about the last block $A_B$ that
implies ${\rm Im}(A_B)\supseteq \{d\} \cup {\rm Im}(A_1) \cup \ldots {\rm Im}(A_{B-1})$ which we refer to
as the {\it last block condition}.

Of the above drawbacks,  (iv) is especially limiting. To illustrate this, we give a few applications where the last block condition,
and hence (iv), does not hold:

\begin{itemize}
    \item \textit{Rank-deficient Quadratic Programming (RDQP).} It is shown in \cite{chen2016direct} that the (non-proximal) ADMM diverges on the following three-block convex RDQP:
    \begin{align*}
    \min_{x_{1},x_{2},x_{3},x_{4}}\  & \frac{1}{2}x_{1}^{2}\\
    \text{s.t.}\  & \left(\begin{array}{cc}
    1 & 1\\
    1 & 1\\
    1 & 1
    \end{array}\right)\left(\begin{array}{c}
    x_{1}\\
    x_{2}
    \end{array}\right)+\left(\begin{array}{c}
    1\\
    1\\
    2
    \end{array}\right)x_{3}+\left(\begin{array}{c}
    1\\
    2\\
    2
    \end{array}\right)x_{4}=0.
    \end{align*}
    \item \textit{Distributed Finite-Sum Optimization (DFSO).} Given a positive integer $B$, 
    consider:
    \begin{equation}
    \min_{x_i \in\rn} \left\{\sum_{t=1}^B (f_t + h_t)(x_t) : x_t-x_B=0,\quad t=1,\ldots,B-1\right\} 
    \label{eq:intro_example}
    \end{equation}
    where $f_i$ is continuously differentiable, $h_t$ is closed convex, and $\nabla f_t$ is Lipschitz continuous for $t=1,...,B$. 
    It is easy to see\footnote{Here, $e_1,\ldots,e_n$ is the standard basis for $\r^{B-1}$, $I_n$ is the $n$-by-$n$ identity matrix, $\textbf{1}\in\r^{B-1}$ is a vector of ones, and $\otimes$ is the Kronecker product of two matrices.} that \eqref{eq:intro_example} is a special case of \eqref{eq:main_prb} where we have $A_s=e_s \otimes I \in \r^{n(B-1)\times n}$ for $s=1,\ldots,B-1$, we have $A_B= -\textbf{1} \otimes I \in \r^{n(B-1)\times n}$, and $d=0$. Moreover, it is straightforward to show that for $s=1,\ldots,B-1$ we have ${\rm Im}(A_s) \cap {\rm Im}(A_B) = 0$ but ${\rm Im}(A_s) \backslash \{0\} \neq \emptyset$, which implies that ${\rm Im}(A_s) \not \subseteq {\rm Im}(A_B)$.

    \item \textit{Decentralized AC Optimal Power Control (DAC-OPF).} 
    The convex version was first considered in \cite{sun2013fully} for the rectangular coordinate formulation, and the problem itself is considered one of the most important ones in power systems decision-making. 
    The nonconvex version of DAC-OPF is a variant  where $h_t$ is the indicator of a convex region given by a finite number of complicated quadratic constraints and $f_t$ is a nonconvex quadratic cost function. 
    A discussion of the limitations induced by assuming any rank condition which implies the last
    block condition is given in \cite{sun2019two}.
\end{itemize}

Our goal in this paper is to develop and analyze the complexity of a proximal ADMM that removes all the drawbacks above. 
For a given $\theta\in(0,1)$, its $k^{{\rm th}}$ iteration
 is based on the \textit{dampened} augmented Lagrangian (AL) function given by
\begin{align}
{\cal L}_{c_k}^{\theta}(x;p) & :=\phi(x)+(1-\theta)\left\langle p,Ax-d\right\rangle +\frac{c_k}{2}\left\Vert Ax-d\right\Vert ^{2},\label{eq:dal}
\end{align}
where $c_k>0$ is the \textit{penalty parameter}. Specifically, it consists of the following updates: given $x^{k-1}=(x^{k-1}_1,\ldots,x^{k-1}_B)$, $p^{k-1}$
$c_k$, $\chi$, and $\lam$, sequentially ($t=1,\ldots,B$) compute the
$t^{\rm th}$ block of $x^k$ as
\begin{align}
x_{t}^{k} & =\argmin_{u_t\in\r^{n_{t}}}\left\{ \lam{\cal L}_{c_k}^{\theta}(\ldots,x_{t-1}^{k},u_t,x_{t+1}^{k-1},\ldots;p^{k-1})+\frac{1}{2}\|u_t-x_{t}^{k-1}\|^{2}\right\} ,\label{eq:x_update}
\end{align}
and then update
\begin{equation}
    p^{k} = (1-\theta)p^{k-1}+\chi c_k\left(Ax^k-d\right),\label{eq:p_update}
\end{equation}
where $\chi\in(0,1)$ is a suitably
chosen under-relaxation parameter. 

\medskip{}

\noindent \emph{Contributions}. 
For proper choices of the stepsize $\lam$ and a non-decreasing sequence of penalty parameters $\{c_k\}_{k\geq 1}$, 
it is shown that if the
Slater-like condition\footnote{Here, $\intr S$ denotes the interior of a set $S$, $\dom\psi$ denotes the domain of a function $\psi$, and $A^{*}$ is the adjoint of linear operator $A$.}
\begin{equation}
\exists z_{\dagger} \in\intr\left(\dom h\right)\text{ such that }A z_{\dagger}=d,\label{eq:slater}
\end{equation}
holds, then DP.ADMM has the following features:
\begin{itemize}
    \item for any tolerance pair $(\rho,\eta)\in\r_{++}^{2}$, it obtains a pair $(\bar{z},\bar{q})$ satisfying  
\begin{equation}
{\rm dist}\left(0,\nabla f(\bar{z})+ A^*\bar{q} + \pt h(\bar{z})\right)\leq\rho,\quad\left\Vert A\bar{z}-d\right\Vert \leq\eta
\label{eq:approx_statn_point}
\end{equation}
in ${\cal O}(\max\{\rho^{-3}, \eta^{-3}\})$ iterations;
    \item it introduces a novel approach for updating the penalty parameter $c_k$, instead of assuming that $c_k=c_1$ for every $k\geq 1$ and that $c_1$ is sufficiently large (such as in \cite{chao2020inertial,jia2021incremental, sun2021dual, jiang2019structured, wang2019global, zhang2020proximal}); 
    \item it does not have any of the drawbacks mentioned in the sentences preceding equation \eqref{eq:intro_example}.
\end{itemize}

\medskip{}

\noindent \emph{Related Works.} Since ADMM-type methods where $f$
is convex have been well-studied in the literature (see, for example,
\cite{bertsekas2016nonlinear,boyd2011distributed,eckstein1992douglas,eckstein1998operator,ruszczynski1989augmented,eckstein1994some,rockafellar1976augmented,gabay1983chapter,glowinski1975approximation,gabay1976dual,eckstein2008family,eckstein2009general,monteiro2013iteration}),
we make no further mention of them here. Instead, we discuss below
ADMM-type
methods where $f$ is nonconvex.

Letting $\delta_{S}$ denote the indicator function of a convex
set $S$ (see Subsection~\ref{subsec:notation}), we first present a list of common assumptions in Table~\ref{tab:assumptions}. 

\begin{table}[!htb]
\begin{centering}
\begin{tabular}{cl}
\toprule 
{\scriptsize{}${\cal Q}$}  & {\scriptsize{}$f(z)=\sum_{t=1}^B f_t(z_t)$ for subfunctions $f_t:{\dom h}_t \mapsto \r$.}
\tabularnewline
{\scriptsize{}${\cal R}_{0}$}  & {\scriptsize{}${\rm Im}(A_ B)\supseteq\{d\}\cup{\rm Im}(A_1)\cup \ldots \cup {\rm Im}(A_{ B-1})$.}
\tabularnewline
{\scriptsize{}${\cal S}$}  & {\scriptsize{}The Slater-like assumption \eqref{eq:slater} holds.}\tabularnewline
{\scriptsize{}${\cal P}$}  & {\scriptsize{}$h_{i}\equiv\delta_{P}$ for $i\in\{1,\ldots, B\}$,
where $P$ is a polyhedral set.}\tabularnewline
{\scriptsize{}${\cal F}$}  & {\scriptsize{}A point $x^0 \in \dom h$
satisfying $Ax^0=d$ is available as an input.}\tabularnewline
\bottomrule
\end{tabular}
\par\end{centering}
\caption{Common nonconvex ADMM assumptions and regularity conditions.
\label{tab:assumptions}}
\end{table}

Earlier developments on
ADMM for solving nonconvex instances of \eqref{eq:main_prb} all assume
that ${\cal R}_0$ hold,
and the ones dealing with
 complexity
establish 
an ${\cal O}(\varepsilon^{-2})$ iteration complexity,
where $\varepsilon := \min\{\rho,\eta\}$.
More specifically,
\cite{goncalves2017convergence, chao2020inertial, themelis2020douglas, wang2019global} present proximal ADMMs under the assumption $B=2$, $h_B \equiv 0$, and  assumption ${\cal Q}$ holds for  \cite{goncalves2017convergence, chao2020inertial, themelis2020douglas}.
Papers \cite{melo2017iterationGauss,melo2017iterationJacobi, jia2021incremental, jiang2019structured} present (possibly linearized) ADMMs under the assumption that $B\geq 2$, $h_B \equiv 0$, and assumption ${\cal Q}$ holds for  \cite{melo2017iterationGauss,melo2017iterationJacobi, jia2021incremental}.  

We next discuss papers that do not assume the restrictive condition ${\cal R}_0$ in Table~\ref{tab:assumptions},
and are based on
ADMM approaches directly applicable to \eqref{eq:main_prb} or some reformulation of it.
An early paper in this direction is
\cite{jiang2019structured}, which establishes
an ${\cal O}(\varepsilon^{-6})$ iteration-complexity bound for an ADMM-type method applied to a  penalty reformulation of \eqref{eq:main_prb} that artificially satisfies ${\cal R}_0$.
On the other hand, development of
ADMM-type methods directly applicable to
\eqref{eq:main_prb} is considerably more challenging
and only a few works have recently surfaced (see Table~\ref{tab:comparisons} below). 

\begin{table}[!htb]
\begin{centering}
\begin{tabular}{ccccccc}
\toprule 
{\scriptsize{}Algorithm } & {\scriptsize{}$\theta$} & {\scriptsize{}$\chi$} & {\scriptsize{}Complexity } & {\scriptsize{}Assumptions} & {\scriptsize{}Adaptive $c$}\tabularnewline
\midrule 
{\scriptsize{}LPADMM \cite{zhang2020proximal} } & {\scriptsize{}$0$} & {\scriptsize{}$(0,\infty)$} & {\scriptsize{}None }  & {\scriptsize{}${\cal P}$,\,${\cal S}$} & {\scriptsize{} No}
\tabularnewline
{\scriptsize{}SDD-ADMM \cite{sun2021dual} } & {\scriptsize{}$(0,1]$} & {\scriptsize{}$[-\frac{\theta}{4},0)$} & {\scriptsize{}${\cal O}(\varepsilon^{-4})$ } & {\scriptsize{}${\cal F}$} & {\scriptsize{} No}
\tabularnewline
\textbf{\scriptsize{}DP.ADMM}{\scriptsize{} }  & {\scriptsize{}$(0,1]$} & {\scriptsize{}$(0,\pi_\theta]$} & {\scriptsize{}${\cal O}(\varepsilon^{-3})$ }  & {\scriptsize{}${\cal S}$} & {\scriptsize{} Yes}
\tabularnewline 
\bottomrule
\end{tabular}
\par\end{centering}
\caption{Comparison of existing ADMM-type methods with DP.ADMM for finding $\varepsilon$-stationary points with $\varepsilon:=\min\{\rho,\eta\}$ and $\pi_\theta = \theta^{2} / [2B(2-\theta)(1-\theta)]$ if $\theta\in(0,1)$ and $\pi_\theta = 1$ if $\theta=1$.  \label{tab:comparisons}}
\end{table}

We now discuss some advantages of DP.ADMM compared to the other two papers in Table~\ref{tab:comparisons}. First, the method in
\cite{sun2021dual} considers a small stepsize (proportional to $\eta^{2}$)
\textit{linearized} proximal gradient update while DP.ADMM considers a large
stepsize (proportional to the inverse of the weak-convexity constant
of $f$) proximal point update as in \eqref{eq:x_update}. 
Second, the method in \cite{sun2021dual} requires a feasible initial point, i.e., a point $z_0 \in \dom h$ satisfying $Az_0 = d$, while DP.ADMM only requires that the initial point be in $\dom h$. 
Third, the methods in \cite{sun2021dual, zhang2020proximal} both require certain hyperparameters (the penalty parameter in \cite{sun2021dual} and an interpolation parameter in \cite{zhang2020proximal}) to be chosen in a range that is hard to compute, while DP.ADMM only requires its main hyperparameter pair $(\chi,\theta)$ to satisfy a simple inequality (see \eqref{eq:chi_theta_cond}). 
Moreover, \cite{sun2021dual} does not specify an easily implementable rule for updating its method's penalty parameter, while DP.ADMM does. 
Fourth, convergence of the method in \cite{zhang2020proximal} requires $h$ being the indicator of a polyhedral set, whereas DP.ADMM applies to any closed convex function $h$. 
Fifth, in contrast to \cite{sun2021dual} and this work, \cite{zhang2020proximal} does not give a complexity bound for its proposed method. Finally, \cite{sun2021dual} considers an unusual negative stepsize for its Lagrange multiplier update  --- which justifies its moniker ``scaled dual descent ADMM'' --- whereas DP.ADMM considers a positive stepsize.

\medskip{}

\noindent \emph{Organization}. Subsection~\ref{subsec:notation}
presents some basic definitions and notation. Section~\ref{sec:admm}
presents the proposed DP.ADMM in two subsections. The first one precisely
describes the problem of interest, while the second one states the static and dynamic 
DP.ADMM variants and their iteration complexities. Section~\ref{sec:static_analysis} and \ref{sec:dynamic_analysis}
present the main properties of the static and dynamic DP.ADMM, respectively. Section~\ref{sec:numerical_experiments} presents some preliminary numerical experiments.
Section~\ref{sec:concluding_remarks} gives some concluding remarks.
Finally, the end of the paper contains several appendices.

\subsection{Notation and Basic Definitions}

\label{subsec:notation}

Let $\r_{+}$ denote the set of nonnegative real numbers, and let $\r_{++}$ denote the set of positive real numbers. Let $\r_{n}$ denote the $n$-dimensional
Hilbert space with inner product and associated norm denoted by $\inner{\cdot}{\cdot}$
and $\|\cdot\|$, respectively. The direct sum (or Cartesian product)
of a set of sets $\{S_{i}\}_{i=1}^{n}$ is denoted by $\prod_{i=1}^{n}S_{i}$.

The smallest positive singular value of a nonzero linear operator
$Q:\r^{n}\to\r^{l}$ is denoted by $\sigma_{Q}^{+}$. For a given
closed convex set $X\subset\r^{n}$, its boundary is denoted by $\partial X$
and the distance of a point $x\in\r^{n}$ to $X$ is denoted by ${\rm dist}_{X}(x)$.
The indicator function of $X$ at a point $x\in\rn$ is denoted by
$\delta_{X}(x)$ which has value $0$ if $x\in X$ and $+\infty$
otherwise. For every $z > 0$ and positive integer $b$, we denote $\log_b^+(z) := \max\{1, \lceil \log_b(z)\rceil\}$.

The domain of a function $h:\r^{n}\to(-\infty,\infty]$ is the set
$\dom h:=\{x\in\r^{n}:h(x)<+\infty\}$. Moreover, $h$ is said to
be proper if $\dom h\ne\emptyset$. The set of all lower semi-continuous
proper convex functions defined in $\r^{n}$ is denoted by $\cConv\rn$.
The set of functions in $\cConv\rn$ which have domain $Z\subseteq\rn$
is denoted by $\cConv Z$. The $\varepsilon$-subdifferential of a
proper function $h:\r^{n}\to(-\infty,\infty]$ is defined by 
\begin{equation}
\partial_{\varepsilon}h(z):=\{u\in\r^{n}:h(z')\geq h(z)+\inner u{z'-z}-\varepsilon,\quad\forall z'\in\r^{n}\}\label{def:epsSubdiff}
\end{equation}
for every $z\in\r^{n}$. The classic subdifferential, denoted by
$\partial h(\cdot)$, corresponds to $\partial_{0}h(\cdot)$. The
normal cone of a closed convex set $C$ at $z\in C$, denoted by $N_{C}(z)$,
is defined as 
\[
N_{C}(z):=\{\xi\in\r^{n}:\inner{\xi}{u-z}\leq\varepsilon,\quad\forall u\in C\}.
\]
If $\psi$ is a real-valued function which is differentiable at $\bar{z}\in\r^{n}$,
then its affine approximation $\ell_{\psi}(\cdot,\bar{z})$ at $\bar{z}$
is given by 
\begin{equation}
\ell_{\psi}(z;\bar{z}):=\psi(\bar{z})+\inner{\nabla\psi(\bar{z})}{z-\bar{z}}\quad\forall z\in\r^{n}.\label{eq:defell}
\end{equation}
If $z=(x,y)$ then $f(x,y)$ is equivalent to
$f(z)=f((x,y))$.

Iterates of a scalar quantity have their iteration number appear as a subscript, e.g., $c_\ell$, while non-scalar quantities have this number appear as a superscript, e.g., $v^k$, and $\hat{p}^\ell$. For variables with multiple blocks, the block number appears as a subscript, e.g., $x_t^{k}$ and $v_{t}^{k}$. Finally, we define the following norm for any quantity $u=(u_1,\ldots,u_B)$ following a block structure as in \eqref{eq:block_structure}:
\begin{equation}
    \|u\|_\dagger = \|(u_1,\ldots,u_B)\|_\dagger := \sum_{t=1}^B \|u_t\|. 
    \label{eq:block_norm}
\end{equation}

\section{Alternating Direction Method of Multipliers}

\label{sec:admm}

This section contains two subsections. The first one precisely describes
the problem of interest and its underlying assumptions, while the second one presents the DP.ADMM and its corresponding iteration complexity.

\subsection{Problem of Interest}

\label{subsec:prb}

This subsection presents the problem of interest and the assumptions
underlying it.

Denote the aggregated quantities 
\begin{gather}
\begin{gathered}
x_{<t}:=(x_{1},\ldots,x_{t-1}), \quad  x_{>t}:=(x_{t+1},\ldots,x_{ B}), \\  x_{\leq t}:=(x_{<t},x_{t}),\quad x_{\geq t}:=(x_{t},x_{>t}),
\end{gathered}
\label{eq:intro_agg}
\end{gather}
for every $x=(x_{1},\ldots,x_{ B})\in {\cal H}$. 
Our problem of interest is finding approximate stationary points of
\eqref{eq:main_prb} under the following assumptions:
\begin{itemize}
\item[(A1)] for every $t = 1,\ldots, B$, we have $h_{t}\in\cConv \r^{n_t}$ and
${\cal H}_t := \dom h_t$ is compact; 
\item[(A2)] $A \not\equiv0$ and ${\cal F}:=\{x\in {\cal H}:Ax=d\}\neq\emptyset$
where ${\cal{H}} := {\cal{H}}_{1} \times \cdots \times {\cal{H}}_{B} $; 
\item[(A3)] $h$ in \eqref{eq:block_structure} is $K_{h}$-Lipschitz continuous on ${\cal H}$ for some $K_{h}\geq0$; 
\item[(A4)] for every $t=1,\ldots,B$,
there exists
$m_{t} \geq 0$
such that
\begin{gather} \label{eq:weak_cvx}
f(x_{<t},\cdot,x_{>t}) + \delta_{{\cal H}_t}(\cdot) + \frac{m_{t}}{2}\|\cdot\|^{2}\text{ is convex for all } x\in{\cal H};
\end{gather}
\item[(A5)] $f$ is differentiable on ${\cal H}$ and, for every $t=1,\ldots,B-1$, there exists $M_{t}\ge 0$ such that
\begin{gather}
\begin{aligned}
& \|\nabla_{x_{t}}f(x_{\leq t},\tilde{x}_{>t})-\nabla_{x_{t}}f(x_{\leq t},{x}_{>t})\|
\leq M_t\|\tilde{x}_{>t}-x_{>t}\| \quad  \forall x,\tilde{x}\in{\cal H};
\end{aligned}
\label{eq:lipschitz_x}
\end{gather}
\item[(A6)] there exists $z_{\dagger} \in{\cal F}$ such that $d_{\dagger}:={\rm dist}_{\pt {\cal H}}(z_{\dagger})>0$. 
\end{itemize}

We now give a few remarks about the above assumptions. 
First, in view of the fact that ${\cal H}$ is compact, the following scalars are bounded:
\begin{equation}
\begin{gathered}D_{\dagger} :=\sup_{z \in {\cal H}}\|z - z_{\dagger}\|,\quad G_{f}:=\sup_{x\in {\cal H}}\|\nabla f(x)\|, \\ \underline{\phi} :=\inf_{x\in {\cal H}} \phi(x), \quad \overline{\phi}:=\sup_{x\in {\cal H}}\phi(x).\end{gathered}
\label{eq:agg_defs}
\end{equation}
Second, if $f$ is a separable function, i.e., it is of the form $f(z)=f_1(z_1) + \cdots + f_B(z_B)$, then each $M_t$ can be chosen to be zero.
Third, any function $h$ given by \eqref{eq:block_structure} such that each $h_t$ for $t=1,\ldots,B$ has the form $h_t =\tilde h_t + \delta_{Z_t}$,
where $\tilde h_t$ is a finite everywhere Lipschitz continuous
convex function
and $Z_t$ is a compact convex set, clearly satisfies condition (A3) for some $K_h$.

For a given tolerance pair
$(\rho,\eta)$, 
we define a
$(\rho,\eta)$-stationary pair
of \eqref{eq:main_prb} as
being a pair
$(\bar z,\bar q) \in {\cal H} \times \r^{\ell}$ satisfying
\eqref{eq:approx_statn_point}.
It is well known
that the first-order necessary condition for a point $z \in {\cal H}$
to be a local minimum of \eqref{eq:main_prb} is that there exists $q\in \r^\ell$ such that
the stationary conditions
\begin{gather*}
0\in\nabla f(z)+A^{*}q+\pt h(z),\quad Az=d
\end{gather*}
hold.
Hence, the requirements in \eqref{eq:approx_statn_point} can be viewed
as a direct relaxation of the above stationary conditions.
For ease of future
reference, we consider the following problem.
\vspace*{1em}
\noindent \begin{center}
{%
{\fboxrule 1.2pt\fboxsep 10pt\fbox{\begin{minipage}[t]{0.85\textwidth}%
\noindent \textbf{Problem~${\cal S}_{\rho,\eta}$}\ :\ Find a $(\rho,\eta)$-stationary pair $(\bar z, \bar q)$ satisfying \eqref{eq:approx_statn_point}.%
\end{minipage}}}} 
\par\end{center}
\vspace*{1em}

We now make three remarks about
Problem~${\cal S}_{\rho,\eta}$.
First, 
$(\bar z, \bar q)$ is a solution of Problem~${\cal S}_{\rho,\eta}$ if and only if there exists a residual $\bar v \in \rn$ such that 
\begin{equation}
\bar{v} \in \nabla f(\bar z) + A^* \bar{q} + \pt h(\bar z), 
\quad \|\bar v\| \leq \rho, \quad \|A\bar z-d\| \leq \eta.
\label{eq:approx_soln}
\end{equation}
Second, condition \eqref{eq:approx_soln} has been considered in many previous works (e.g., see \cite{kong2021aidal,kong2021thesis,kong2019complexity,kong2020efficient,melo2020iteration}).
Third, in the case where $\|\cdot\|=\|\cdot\|_2$ and $\rho=\eta$, the stationarity condition in \eqref{eq:approx_statn_point} implies the stationarity condition of the papers \cite{jiang2019structured, sun2021dual} in Table~\ref{tab:comparisons}. Specifically, \cite[Definition~3.6]{jiang2019structured} and \cite[Definition~3.3]{sun2021dual} consider a pair $(z,q)\in{\cal H}\times \r^{\ell}$ to be an $\varepsilon$-stationary pair if it satisfies
\[
{\rm dist}(0, \nabla_{z_t} f(z_1,\ldots,z_B) + A_t^* q + \pt h_t(z_t)) \leq \varepsilon, \quad \|Az-d\|\leq \varepsilon,
\]
for every $t=1,\ldots,B$.

In the following subsection, we present a method (Algorithm~\ref{alg:static_dp_admm}) that computes a triple $(\bar{z}, \bar{q}, \bar{v})$ satisfying \eqref{eq:approx_soln}, and hence
which guarantees that
$(\bar z, \bar q)$ is a solution of Problem~${\cal S}_{\rho,\eta}$.

\subsection{DP.ADMM}

\label{subsec:dp_admm}

We present DP.ADMM in two parts. The first part presents a static version of DP.ADMM which either (i) stops with a solution of Problem~${\cal S}_{\rho,\eta}$ or (ii)  signals that its penalty parameter is too small. 
The second part presents the (dynamic) DP.ADMM that repeatedly invokes the static version on an increasing sequence of penalty parameters. 

Both versions of DP.ADMM make use of the following condition on $(\chi,\theta)$:
\begin{equation}
2 \chi B (2-\theta)(1-\theta) \le \theta^{2}, \quad (\chi,\theta)\in(0,1]^2.\label{eq:chi_theta_cond}
\end{equation}
For ease of reference and discussion, the pseudocode for the static DP.ADMM is given in Algorithm~\ref{alg:static_dp_admm} below. 
Notice that
the classic proximal ADMM iteration 
\begin{align*}
x_{t}^{k} & =\argmin_{u^{t}\in\r^{n_{t}}}\left\{ \lam{\cal L}_{c}^{0}(x_{<t}^{k},u_{t},x_{>t}^{k-1};p^{k-1})+\frac{1}{2}\|u_{t}-x_{t}^{k-1}\|^{2}\right\},\quad  t=1,\ldots, B,\\
p^{k} & = p^{k-1}+c\left(Ax^{k}-d\right),
\end{align*}
corresponds to the case of $(\chi, \theta)=(1,0)$, where $c \ge 1$ is a fixed penalty parameter.

\begin{algorithm}[!htb]
\caption{Static DP.ADMM}
\label{alg:static_dp_admm}
\begin{algorithmic}[1]
\StateO \texttt{Input}: {${x}^{0}\in {\cal H}$,  $p^{0}\in A(\rn)$, $\lam\in(0,1/(2m)]$}, $c>0$;
\StateO \texttt{Require}: $m$ as in \eqref{eq:pp_indep_defs}, $(\rho,\eta) \in \r^2_{++}$, $(\chi,\theta)$ as in \eqref{eq:chi_theta_cond}
\vspace*{0.5em}
\For{$k\gets1,2,\ldots$}
\LineComment \textcolor{green}{\hspace*{1.5em}\texttt{STEP 1} (prox update)}:
\For{$t\gets1,2,\ldots, B$}
\State $x_{t}^{k}\gets\argmin_{u_t\in\r^{n_{t}}}\left\{ \lam{\cal L}_{c}^{\theta}(x_{<t}^{k},u_t,x_{>t}^{k-1};p^{k-1})+\frac{1}{2}\|u_t-x_{t}^{k-1}\|^{2}\right\}$
\EndFor 

\State ${q}^{k}\gets(1-\theta)p^{k-1}+c(Ax^{k}-d)$
\LineComment \textcolor{green}{\texttt{STEP 2a} (successful termination check)}:
\For{$t\gets1,2,\ldots, B$}
\State $\delta_t^k \gets \nabla_{x_{t}}f(x_{\leq t}^{k}, x_{>t}^{k})-\nabla_{x_{t}}f(x_{\leq t}^{k},x_{>t}^{k-1})$
\State $v_{t}^{k}\gets \delta_t^k + cA_{t}^{*}\sum_{s=t+1}^{ B}A_{s}(x_{s}^{k}-x_{s}^{k-1})-\frac{1}{\lam}(x_{t}^{k}- x_{t}^{k-1})$ \;

\EndFor 

\If{$\|v^{k}\|\leq\rho$ \textbf{ and }$\|Ax^{k}-d\|\leq\eta$}
\State \Return{$(x^{k},p^{k},{q}^{k},v^{k})$}\label{ln:terminate} 
\EndIf 
\LineComment \textcolor{green}{\texttt{STEP 2b} (unsuccessful termination check)}: \;
\If{$k \equiv 0 \bmod 2$ and $k\geq 3$}
\State ${\cal S}_k^{(v)} \gets \frac{2}{k+2}\sum_{i=k/2}^k\|v^{i}\|$ \;
\State ${\cal S}_k^{(f)} \gets \frac{2}{k+2}\sum_{i=k/2}^k\|Ax^{i}-d\|$ \;

\If{$\frac{1}{\rho} \cdot {\cal S}_{k}^{(v)} + \frac{1}{\eta}\sqrt{\frac{c^3}{ k}} \cdot {\cal S}_{k}^{(f)} \leq 1$}

\State \Return{$(x^{k},p^{k},{q}^{k},v^{k})$}\label{ln:end_cycle}
\EndIf 
\EndIf 
\LineComment \textcolor{green}{\texttt{STEP 3} (multiplier update)}:
\State $p^{k}\gets(1-\theta)p^{k-1}+\chi c(Ax^{k}-d)$\;

\EndFor 

\end{algorithmic}
\end{algorithm}

The next result describes the iteration complexity and
some useful technical properties of Algorithm~\ref{alg:static_dp_admm}. Its proof is given in Section~\ref{sec:prop_prf}, and it uses three sets of scalars. The first set is independent of $(c,p^0)$ and is given by
\begin{gather}
\begin{gathered}
M:=\max_{1\leq t\le B}M_{t}, \quad m :=\max_{1\leq t\leq B} m_{t}, 
\quad \Delta_\phi := \overline{\phi}-  \underline{\phi}, \quad  \kappa_{0} := \frac{2B^2\left({\lam} M + 1 \right)}{\sqrt{\lam}}, \\
\kappa_{1} := \frac{\chi\|A\|D_{\dagger}}{\theta}, \quad
\kappa_{2} := \frac{1}{\theta}\left[1 + \frac{2\chi D_{\dagger} (K_{h}+G_{f})}{\theta d_{\dagger}\sigma_{A}^{+}}\right] + 1, \\
\kappa_3 := \frac{108\kappa_2^2}{\chi^2}, \quad \kappa_4 := \frac{\theta d_{\dagger}\sigma_{A}^{+}}{\chi D_{\dagger}}, \quad \kappa_5 := 8(B-1)\|A\|^2_\dagger, \quad \kappa_6 :=  3 + \frac{8 \kappa_0^2 \Delta_{\phi}}{\kappa_4^2}.
\end{gathered}
\label{eq:pp_indep_defs}
\end{gather}
where $(G_f, D_{\dagger}, \overline{\phi}, \underline{\phi})$, $K_h$, and $(m_t,M_t)$ are as in \eqref{eq:agg_defs}, (A3), and (A4). The second set is dependent on a given lower bound $\underline{c}$ on $c$ and is given by
\begin{align}
\begin{gathered}
\tilde{\kappa}_{\underline{c}}^{(0)} := 2\left(\sqrt{\Delta_{\phi}}+\frac{5\kappa_2}{\chi\sqrt{\underline{c}}}\right), \quad
\tilde{\kappa}_{\underline{c}}^{(1)}  := 3 \kappa_5 [\tilde{\kappa}_{\underline{c}}^{(0)}]^2, \quad
\tilde{\kappa}_{\underline{c}}^{(2)} := 3\kappa_0^2 [\tilde{\kappa}_{\underline{c}}^{(0)}]^2.
\end{gathered}
\label{eq:tilde_kappa_defs}
\end{align}
The third set is dependent on a given upper bound $\cal R$ on $\|p^0\|/c$ and is given by
\begin{align}
\begin{aligned}
\xi_{{\cal R}}^{(0)} &:= \frac{8}{\kappa_4^2} \left[\frac{9\kappa_0^2({\cal R} +\kappa_1)^2}{\chi^2} + \kappa_5 \Delta_\phi \right]  + (1-\theta)({\cal R}+\kappa_1), \\
\xi_{{\cal R}}^{(1)} &:= \frac{72\kappa_5({\cal R} + \kappa_1)^2}{\chi^{2}\kappa_4^2}.
\end{aligned}
\label{eq:xi_defs}
\end{align}
\begin{prop}
\label{prop:cycle_props} 
Let ${\cal R} \ge 0$ and $\underline{c} >0$ be given, and assume
that the pair $(c,p^0)$ given to
Algorithm~\ref{alg:static_dp_admm}
satisfies 
\begin{equation}
\|p_0\| \le c {\cal R}, \quad c \ge \underline{c}. \label{eq:R_cbar_bd}
\end{equation}
Then, the following statements hold about the call to Algorithm~\ref{alg:static_dp_admm}:
\begin{itemize}
\item[(a)] it terminates in a number of iterations bounded by
\begin{align}
{\cal T}_c (\rho,\eta\,|\,\underline{c}, {\cal R}) & := 48\left( \left\{\kappa_6 + \frac{\tilde{\kappa}_{\underline{c}}^{(1)}}{\rho^2}\right\} +  
\left\{\xi_{{\cal R}}^{(0)} + \frac{\kappa_3}{\eta^2} +  \frac{\tilde{\kappa}_{\underline{c}}^{(2)}}{\rho^2} \right\} c + 
\xi_{{\cal R}}^{(1)} c^2  \right),
\label{eq:T_static_def}
\end{align}
where $(\kappa_3,\kappa_6)$, $(\tilde{\kappa}_{\underline{c}}^{(1)}, \tilde{\kappa}_{\underline{c}}^{(2)})$, and $(\xi_{\cal R}^{(0)}, \xi_{\cal R}^{(1)})$ are
as in \eqref{eq:pp_indep_defs}, \eqref{eq:tilde_kappa_defs}, and \eqref{eq:xi_defs}, respectively;

\item[(b)] if it terminates successfully in Step~2a, then the first and third components of its output quadruple $(\bar{z}, \bar{p}, \bar{q}, \bar{v})$ solve
Problem~${\cal S}_{\rho,\eta}$;
\item[(c)] if $c$ satisfies 
\begin{equation}
c \geq \hat{c}(\rho,\eta\,|\,\underline{c},{\cal R}) := \frac{1}{\underline{c}^2}\left[{\cal T}_{\underline{c}}(1,1\,|\,\underline{c},{\cal R}) + \frac{\sqrt{\underline{c}^3 \cdot {\cal T}_{\underline{c}}(1,1\,|\,\underline{c},{\cal R})}}{\min\{\rho,\eta\}} \right], \label{eq:penalty_thresh}
\end{equation}
where ${\cal T}_c(\rho,\eta\,|\,\underline{c},{\cal R})$ is as in (a), then it must terminate successfully.
\end{itemize}
\end{prop}

We now make some remarks about
Proposition~\ref{prop:cycle_props}.
First, statement (c) implies that Algorithm~\ref{alg:static_dp_admm} terminates successfully if its penalty parameter $c$ is sufficiently large,
i.e., $c=\Omega(\varepsilon^{-1})$ where
$\varepsilon := \min \{ \rho,\eta\}$.
Moreover, if a penalty parameter $c$
satisfying \eqref{eq:penalty_thresh} and the condition that
$c={\cal O}(\varepsilon^{-1})$ is known,
then it follows from Proposition~\ref{prop:cycle_props}(a) that the 
iteration complexity of Algorithm~\ref{alg:static_dp_admm}
for finding a solution of Problem~${\cal S}_{\rho,\eta}$ is
${\cal O}(\varepsilon^{-3})$.

Since a penalty parameter $c$ as in the above paragraph is nearly impossible to compute, we next present an adaptive method, namely, Algorithm~\ref{alg:dp_admm} below, which adaptively increases the penalty parameter $c$, and whose overall number of iterations is also ${\cal O}(\varepsilon^{-3})$.



\begin{algorithm}[!htb]
\caption{DP.ADMM}
\label{alg:dp_admm}
\begin{algorithmic}[1]
\StateO \texttt{Input}: $\bar{z}^0\in {\cal H}$,  {$\lam\in(0,1/(2m)]$}, $c_1 >0$
\StateO \texttt{Require}: $m$ as in \eqref{eq:pp_indep_defs}, $(\rho,\eta)\in (0,1)^2$,
 $(\chi,\theta)$ as in \eqref{eq:chi_theta_cond}
\vspace*{0.5em}
\State $\bar{p}^{0} \gets 0$
\For{$\ell\gets 1,2,\ldots$}
\State \multiline{%
\textbf{call} Algorithm~\ref{alg:static_dp_admm} with inputs $(x^{0},p^{0},\lam,c)=(\bar{z}^{\ell-1},\bar{p}^{\ell-1},\lam,c_\ell)$ and parameters $m$, $(\rho, \eta)$, and $(\chi, \theta)$ to obtain an output quadruple $(\bar{z}^\ell, \bar{p}^\ell, \bar{q}^\ell, \bar{v}^\ell)$}
\If{$ \|\bar{v}^\ell\|\leq\rho$ \textbf{ and } $\|A\bar{z}^{\ell}-d\|\leq \eta$}
\State \Return{$(\bar{z}^\ell, \bar{q}^\ell)$}
\EndIf
\State $c_{\ell+1} \gets 2 c_\ell$
\EndFor
\end{algorithmic}
\end{algorithm}

{Some comments about Algorithm~\ref{alg:dp_admm} are in order. 
First, it employs a ``warm-start'' type strategy for calling Algorithm~\ref{alg:static_dp_admm} at each iteration $\ell$. Specifically, the input
of
the $\ell^{\rm th}$ to Algorithm~\ref{alg:static_dp_admm} is the pair $(\bar z^{\ell-1}, \bar p^{\ell-1})$ output by the previous call to Algorithm~\ref{alg:static_dp_admm}.
Second, the initial penalty parameter
${c_1}$ can be chosen to be any positive scalar, in contrast to many of the methods listed in Section~\ref{sec:intro} where this parameter must be chosen sufficiently large. 
Third, the initial point $\bar{z}^0$ only needs to be in the domain of $h$ and need not be feasible or near feasible. 
Finally, while the initial Lagrange multiplier $\bar{p}^0$ is chosen to be zero, the analysis in this paper can be carried out for any $\bar{p}^0\in A(\rn)$, at the cost of more complicated complexity bounds.}

The next result, whose proof is given in Section~\ref{sec:dynamic_analysis}, gives the complexity of Algorithm~\ref{alg:dp_admm} in terms of the total number of iterations of Algorithm~\ref{alg:static_dp_admm} across all of its calls. 

\begin{thm}
\label{thm:total_compl} 
Define the scalars
\begin{equation}
T_1 := {\cal T}_{c_1}(1,1\,|\,c_1, 2\kappa_1), \quad \varepsilon := \min\{\rho,\eta\}, \label{eq:dyn_aliases}
\end{equation}
where $\kappa_1$ and ${\cal T}_{c}(\cdot, \cdot\,|\, \cdot, \cdot)$ are as in \eqref{eq:pp_indep_defs} and \eqref{eq:T_static_def}, respectively. 
Then, Algorithm~\ref{alg:dp_admm} stops and outputs a pair that solves Problem~${\cal S}_{\rho,\eta}$ in a number of iterations  of Algorithm~\ref{alg:static_dp_admm} bounded by
\begin{align}
T_1 \left( 2E_{0}^{2} +  \frac{E_0 + 2E_{1}^2}{\varepsilon^{2}}+\frac{E_{1}}{\varepsilon^{3}}\right)
\label{eq:gen_compl}
\end{align}
where
\begin{equation}
E_{0} := 2 \left(1 + \frac{T_1^2}{c_1^3}\right), \quad E_1 := 2\sqrt{\frac{T_1}{c_1^{3}}}.
\label{eq:E_defs}
\end{equation}
\end{thm}

Since $T_1 = {\cal O}(c_1^{-1})$ in view of \eqref{eq:T_static_def} and \eqref{eq:dyn_aliases}, it follows from \eqref{eq:gen_compl} and \eqref{eq:E_defs} that if $c_1^{-1} = {\cal O}(1)$, then the overall complexity of Algorithm~\ref{alg:dp_admm} is ${\cal O}(\varepsilon^{-3})$.

\section{Analysis of Algorithm~\ref{alg:static_dp_admm}}

\label{sec:static_analysis}

This section presents the main properties of Algorithm~\ref{alg:static_dp_admm}, and it contains
three subsections. More specifically, the first
(resp., second) subsection establishes some key bounds on
the ergodic means of the sequences $\{\|v^k\|\}_{k\geq 0}$ and
$\{\|A x^k-d\|\}_{k\geq 0}$
(resp., the sequence $\{\|p_k\|\}_{k\geq 0}$).
The third one proves Proposition~\ref{prop:cycle_props}.

Throughout this section, we let $\{(v^{i}, x^{i},p^{i}, {q}^{i})\}_{i=1}^{k}$ denote the iterates generated by Algorithm~\ref{alg:static_dp_admm} up to and including the $k^{{\rm th}}$
iteration for some $k\geq 3$. Moreover, for every $i\geq1$
and $(\chi,\theta)\in\r_{++}^{2}$ satisfying \eqref{eq:chi_theta_cond}, we make use of the following
useful constants and shorthand notation 
\begin{equation}
\begin{gathered}a_{\theta}=\theta(1-\theta),\quad b_{\theta}:=(2-\theta)(1-\theta), \\
\gamma_{\theta}:=\frac{(1-2 B\chi b_{\theta})-(1-\theta)^{2}}{2\chi}, \quad
 f^{i}:=Ax^{i}-d,
\end{gathered}
\label{eq:global_admm_consts}
\end{equation}
the aggregated quantities in \eqref{eq:intro_agg}, and
the averaged quantities 
\begin{gather}
S_{j,k}^{(p)} :=\frac{\sum_{i=j}^{k}\|p^{i}\|}{k-j+1},\quad
S_{j,k}^{(v)} :=\frac{\sum_{i=j}^{k}\|v^{i}\|}{k-j+1}, \quad
S_{j,k}^{(f)} :=\frac{\sum_{i=j}^{k}\|f^{i}\|}{k-j+1}. \label{eq:aux_admm_defs}
\end{gather}
for every $j=1,\ldots, k$. Notice that $\gamma_\theta \geq \theta / \chi$ in view of \eqref{eq:chi_theta_cond}. We also denote $\Delta y^{i}$ to be the difference
of iterates for any variable $y$ at iteration $i$, i.e.,
\begin{equation}
\Delta y^{i} \equiv y^{i}-y^{i-1}. \label{not:Delta_var}
\end{equation}

\subsection{Properties of the Key Residuals}
\label{subsec:basic}

This subsection presents bounds on the residuals $\{\|v^{i}\|\}_{i=2}^k$ and $\{\|f^{i}\|\}_{i=2}^k$ generated by Algorithm~\ref{alg:static_dp_admm}. These bounds will be particularly helpful for proving Proposition~\ref{prop:cycle_props} in Subsection~\ref{sec:prop_prf}.

The first result presents some key properties about the generated iterates.

\begin{lem}
\label{lem:refine_props}For $i=1,\ldots,k$, 
\begin{itemize}
\item[(a)] $f^{i}=\left[p^{i}-(1-\theta)p^{i-1}\right]/(\chi c)$;
\item[(b)] $v^{i}\in\nabla f(x^{i})+A^{*}{q}^{i}+\pt h(x^{i})$ and
{
\begin{equation}
\|v^{i}\|\leq B \left(M + \frac{1}{\lam} \right)\|\Delta x^{i}\|_\dagger +  
c \|A\|_\dagger \sum_{t=2}^B \|A_t \Delta x_t^i\|, \label{eq:v_bd}
\end{equation}
where $\|\cdot\|_\dagger$ is as in \eqref{eq:block_norm}.}
\end{itemize}
\end{lem}

\begin{proof}
(a) This is immediate from step~3 of Algorithm~\ref{alg:static_dp_admm} and the definition
of $f^{i}$ in \eqref{eq:global_admm_consts}.

(b) We first prove the required inclusion. The optimality of $x_{t}^{k}$
in Step~1 of Algorithm~\ref{alg:static_dp_admm}, and assumption (A4),
imply that 
\begin{align*}
0 & \in\pt\left[{\cal L}_{c}^{\theta}(x_{<t}^{i},\, \cdot \,,x_{>t}^{i-1};p^{i-1})+\frac{1}{2\lam}\|\cdot-x_{k}^{i-1}\|^{2}\right](x^{i})\\
 & =\nabla_{x_{t}}f(x_{\leq t}^i,x_{>t}^{i-1})+A_{t}^{*}\left[(1-\theta)p^{i-1}+c[A(x_{\leq t}^i,x_{>t}^{i-1})-d]\right]+\pt h_{t}(x_{t}^{i})+\frac{1}{\lam}\Delta x_{t}^{i}\\
 & =\nabla_{x_{t}}f(x_{\leq t}^i,x_{>t}^{i-1})+A_{t}^{*}\left({q}^{i} - c\sum_{s=t+1}^{ B}A_s \Delta x_{s}^{i}\right)+\pt h_{t}(x_{t}^{i})+\frac{1}{\lam}\Delta x_{t}^{i}\\
 & =\nabla_{x_{t}}f(x^{i})+A_{t}^{*}{q}^{i}+\pt h_{t}(x_{t}^{i})-v_{t}^{i}.
\end{align*}
for every $1\leq t\leq B$. Hence, the inclusion holds.
To show the inequality, let $1\leq t\leq B$ be fixed
and use the triangle inequality, the definition of $v_{t}^{i}$, and
assumption (A5) to obtain 
\begin{align*}
\|v_{t}^{i}\| & \leq\|\nabla_{x_{t}}f(x_{\leq t}^{i}, x_{>t}^{i})-\nabla_{x_{t}}f(x_{\leq t}^i,x_{>t}^{i-1})\|+c\sum_{s=t+1}^{ B}\|A_{t}^{*}A_{s}\Delta x_{s}^{i}\|+\frac{1}{\lambda}\|\Delta x_{t}^{i}\|\\
 & \leq M_t \|x_{>t}^{i}-x_{>t}^{i-1}\|+c \|A_t\| \sum_{s=t+1}^{ B}\|A_{s}\Delta x_{s}^{i}\|+ \frac{1}{\lam}\|\Delta x_{t}^{i}\|\\
 & \leq 
 \left(M+ \frac{1}{\lam}\right)\sum_{s=t}^{ B}\|\Delta x_{s}^{i}\|+ c\|A_t\| \sum_{t=2}^B \|A_t \Delta x_t^i\|. 
\end{align*}
Summing the above bound from $t=1$ to $B$, and using the definition of $M$ in \eqref{eq:pp_indep_defs} and the triangle inequality, we conclude that 
\begin{align*}
\|v^{i}\| & \leq\sum_{t=1}^{ B}\|v_{t}^{i}\|
 \leq\left(M+ \frac{1}{\lam}\right)\sum_{t=1}^{ B}\sum_{s=t}^{ B}\|\Delta x_{s}^{i}\|+c\|A\|_\dagger \sum_{t=2}^B \|A_t \Delta x_t^i\| \\
 &\leq B\left(M+ \frac{1}{\lam}\right)\|\Delta x^{i}\|_\dagger + c\|A\|_\dagger \sum_{t=2}^B \|A_t \Delta x_t^i\|.
\end{align*}
\end{proof}

Notice that part (b) of the above result implies that $(\bar{x},\bar{v},\bar{p})=(x^{i},v^{i},{q}^{i})$ satisfies the inclusion in \eqref{eq:approx_soln}. 
Hence, if $\|v^{i}\|$ and $\|f^{i}\|$ are sufficiently small at some iteration $i$, then Algorithm~\ref{alg:static_dp_admm} clearly returns a solution of Problem~${\cal S}_{\rho,\eta}$ at iteration $i$, i.e., Proposition~\ref{prop:cycle_props}(b) holds. 
However, to understand when Algorithm~\ref{alg:static_dp_admm} terminates, we will need to develop  more refined bounds on $\|v_i\|$ and $\|f_i\|$.

To begin, we present some relations between the perturbed augmented Lagrangian ${\cal L}_c^\theta(\cdot;\cdot)$ and the iterates $\{(x^i, p^i)\}_{i=1}^k$. For conciseness, its proof is given in Appendix~\ref{app:tech_ineq}.

\begin{lem}
\label{lem:key_iter_relations}
For $i=1,\ldots,k$,
\begin{itemize}
\item[(a)] ${\cal L}_{c}^{\theta}(x^{i};p^{i})-{\cal L}_{c}^{\theta}(x^{i};p^{i-1}) = b_{\theta}\|\Delta p^{i}\|^{2} / (2\chi c) + a_{\theta} \left(\|p^{i}\|^{2}-\|p^{i-1}\|^{2}\right) / (2\chi c);$
\item[(b)] {${\cal L}_{c}^{\theta}(x^{i};p^{i-1})-{\cal L}_{c}^{\theta}(x^{i-1};p^{i-1}) 
\leq
- \|\Delta x^{i}\|^{2}/(2\lam) - c \sum_{t=1}^{ B}\|A_{t}\Delta x_{t}^{i}\|^{2} / 2;$}
\item[(c)] if $i\geq2$, it holds that
\begin{equation}
\frac{b_{\theta}}{2\chi c}\|\Delta p^{i}\|^{2}-\frac{c}{4}\sum_{t=1}^{ B}\|A_{t}\Delta x_{t}^{i}\|^{2}\leq\frac{\gamma_{\theta}}{4 B\chi c}\left(\|\Delta p^{i-1}\|^{2}-\|\Delta p^{i}\|^{2}\right).\label{eq:acc_resid_bd}
\end{equation}
\end{itemize}
\end{lem}

The next result uses the above relations to establish a bound on the quantities in the right-hand-side of \eqref{eq:v_bd}.

\begin{lem}
\label{lem:main_descent}
For $j = 1,\ldots, k$,
\begin{align}
& \sum_{i=j+1}^{k} \|v^i\|^2
 \leq
(\kappa_0^2 + \kappa_5 c) \left[ \Psi_{j}(c) - \Psi_{k}(c) \right],
\label{eq:poten_resid_bd}
\end{align}
where $(\kappa_0, \kappa_5)$ is as in \eqref{eq:pp_indep_defs}, and denoting $(a_\theta, \gamma_\theta)$ and  as in \eqref{eq:global_admm_consts}, we have
\begin{align}
\Psi_{i}(c) & :={\cal L}_{c}^{\theta}(x^{i};p^{i})-\frac{a_{\theta}}{2\chi c}\|p^{i}\|^{2}+\frac{\gamma_{\theta}}{4 B\chi c}\|\Delta p^{i}\|^{2}\quad\forall i\geq1.
\end{align}
\end{lem}

\begin{proof}
Using the inequality $\|z\|_1^2 \leq n\|z\|_2^2$ for $z\in\rn$ and \eqref{eq:v_bd}, we first have that
\begin{align}
\sum_{i=j+1}^{k}\|v^{i}\|^{2} & \overset{\eqref{eq:v_bd}}{\leq}\sum_{i=j+1}^{k}\left[B\left(M + \frac{1}{\lambda}\right)\|\Delta x^{i}\|_{\dagger}+ c\|A\|_\dagger \sum_{t=2}^B \|A_t \Delta x_t^i\| \right]^{2}\nonumber \\
 & \leq\sum_{i=j+1}^{k}2B^{2}\left(M + \frac{1}{\lambda}\right)^{2}\|\Delta x^{i}\|_{\dagger}^{2}+ c^2\|A\|_\dagger^2 \left(\sum_{t=2}^B \|A_t \Delta x_t^i\|\right)^2 \nonumber \\
 & \leq\sum_{i=j+1}^{k}2B^{4}\left(M+\frac{1}{\lambda}\right)^{2}\|\Delta x^{i}\|^{2}+2 (B-1) c^{2} \|A\|_\dagger^2 \sum_{t=2}^{B}\|A_{t}\Delta x_{t}^{i}\|^{2}\nonumber \\
 & \leq (\kappa_0^2 + \kappa_5 c) \sum_{i=j+1}^{k}\left[\frac{1}{2\lambda}\|\Delta x^{i}\| +\frac{c}{4}\sum_{t=2}^{B}\|A_{t}\Delta x_{t}^{i}\|^{2}\right].
 \label{eq:v_sqr_tech_bd}
\end{align}
Combining Lemma~\ref{lem:key_iter_relations}(a)--(c),
 the definition of $\Psi_{\theta}^{i}$, and the bound $(a+b)^2\leq 2a^2 + 2b^2$ for $a,b\in\r_+$, we also have that
\begin{align*}
& \frac{1}{2\lam} \|\Delta x^{i}\|^{2}+\frac{c}{4}\sum_{t=2}^{ B}\|A_{t}\Delta x^{i}\|^{2} \\
 \overset{\text{L.\ref{lem:key_iter_relations}(a)-(b)}}{\leq}
 & {\cal L}_c^\theta (x^{j-1};p^{j-1}) - {\cal L}_c^\theta (x^{j};p^{j}) + 
  \frac{a_{\theta}}{2\chi c} \Delta_{p,j}^{(2)} +\frac{b_{\theta}}{2\chi c}\|\Delta p^{i}\|^{2}-\frac{c}{4} \sum_{t=1}^{ B}\|A_{t}\Delta x_{t}^{i}\|^{2}\\
  \overset{\text{L.\ref{lem:key_iter_relations}(c)}}{\leq}
&  {\cal L}_c^\theta (x^{j-1};p^{j-1}) - {\cal L}_c^\theta (x^{j};p^{j}) +
\frac{a_{\theta}}{2\chi c} \Delta_{p,j}^{(2)} +\frac{\gamma_{\theta}}{4 B\chi c}\left(\|\Delta p^{i-1}\|^{2}-\|\Delta p^{i}\|^{2}\right)\\
  = \ \ \ & \Psi_{i-1}(c)-\Psi_{i}(c),
\end{align*}
where $\Delta_{p,j}^{(2)} := \|p^{j}\|^2 - \|p^{j-1}\|^2$.
Consequently, summing the above inequality from $i=j+1$ to $k$, and combining the resulting inequality with \eqref{eq:v_sqr_tech_bd}, yields the desired bound. 
\end{proof}

We now bound the quantity on the right-hand-side of \eqref{eq:poten_resid_bd}
\begin{lem} \label{lem:Lagr_bds}
For any $j\geq 1$ and $k\geq 1$, 
\begin{itemize}
    \item[(a)] ${\cal L}_{c}^{\theta}(x^{j};p^{j})\leq \phi(x^{j}) + 3(\|p^j\|^2 + \|p^{j-1}\|^2)/(\chi^2 c)$; 
    \item[(b)] ${\cal L}_{c}^{\theta}(x^{k};p^{k})\geq \phi(x^{k})-\|p^{k}\|^{2}/(2c)$;
    \item[(c)] it holds that
    \begin{equation}
        \label{eq:poten_upper_bd}
        \Psi_j(c) - \Psi_k(c) 
        \leq  
        \Delta_\phi + 4\left(\frac{\|p^{j}\|^{2}+\|p^{j-1}\|^{2}+\|p^{k}\|^{2}}{\chi^{2}c}\right),
    \end{equation}
    where $\Psi_i(\cdot)$ and $\Delta_\phi$ are as in \eqref{eq:poten_resid_bd} and \eqref{eq:pp_indep_defs}, respectively.
\end{itemize}
\end{lem}

\begin{proof} 
(a)--(b) See Appendix~\ref{app:tech_ineq}.

(c) Using parts (a)--(b), the fact that $a_\theta \in (0,1)$ and $(\chi,\theta)\in(0,1)^{2}$, the relation $(a+b)^2\leq 2a^2 + 2b^2$ for $a,b\in \r_+$, and the bound
$\gamma_{\theta}\leq1/(2\chi)$, it holds that
\begin{align*}
& \Psi_{j}(c)-\Psi_{k}(c)\\
&=\left[{\cal L}_{c}^{\theta}(x^{j};p^{j})-{\cal L}_{c}^{\theta}(x^{k};p^{k})\right]+\frac{a_{\theta}(\|p^{k}\|^{2}-\|p^{j}\|^{2})}{2\chi c}+\frac{\gamma_{\theta}(\|\Delta p^{j}\|^{2}-\|\Delta p^{k}\|^{2})}{4B\chi c}\\
&\leq \left[{\cal L}_{c}^{\theta}(x^{j};p^{j})-{\cal L}_{c}^{\theta}(x^{k};p^{k})\right]+\frac{a_{\theta}\|p^{k}\|^{2}}{2\chi c}+\frac{\gamma_{\theta}\|\Delta p^{j}\|^{2}}{4B\chi c}\\
 & \leq\left[{\cal L}_{c}^{\theta}(x^{j};p^{j})-{\cal L}_{c}^{\theta}(x^{k};p^{k})\right]+\frac{\|p^{k}\|^{2}}{2\chi c}+\frac{\|p^{j-1}\|^{2} + \|p^{j}\|^{2}}{4B\chi^{2}c}\\
 & \overset{\text{(a)-(b)}}{\leq}\left[\phi(x^{j})-\phi(x^{k})+\frac{3(\|p^{j}\|^{2}+\|p^{j-1}\|^{2})}{\chi^2 c}+\frac{\|p^{k}\|^{2}}{2c}\right] + \\
 & \qquad \quad \frac{\|p^{k}\|^{2}}{2\chi c}+\frac{\|p^{j-1}\|^{2} + \|p^{j}\|^{2}}{4B\chi^{2}c} \leq\Delta_{\phi} + 4\left(\frac{\|p^{j}\|^{2}+\|p^{j-1}\|^{2}+\|p^{k}\|^{2}}{\chi^{2}c}\right).
\end{align*}
\end{proof}

The next result presents bounds on $S_{j+1,k}^{(f)}$ and $S_{j+1,k}^{(v)}$.

\begin{prop}
\label{prop:cross_mult_bd} For $j = 1,\ldots, k-1$,
\begin{align}
S_{j+1,k}^{(f)} & \leq \frac{\|p^{j}\|+2S_{j+1,k}^{(p)}}{\chi c}, 
\label{eq:main_fk_bd} \\
S_{j+1,k}^{(v)} & \leq 2\sqrt{\frac{\kappa_0^2 + \kappa_5 c}{k-j}} \left( \Delta_\phi^{1/2} + \frac{\|p^{j}\|+\|p^{j-1}\|+\|p^{k}\|}{\chi\sqrt{c}} \right),
\label{eq:main_vk_bd}
\end{align}
where $(\kappa_0, \kappa_5, \Delta_\phi)$ is as in \eqref{eq:pp_indep_defs}.
\end{prop}

\begin{proof}
Using Lemma~\ref{lem:refine_props}(a), the fact that $\theta\in(0,1)$, and the triangle inequality, it holds that 
\begin{align*}
S_{j+1,k}^{(f)}=\frac{\sum_{i=j+1}^{k}\|p^{i}-(1-\theta)p^{i-1}\|}{\chi c(k-j)}\leq\frac{\sum_{i=j+1}^{k}(\|p^{i-1}\|+\|p^{i}\|)}{\chi c(k-j)}\leq\frac{\|p^{j}\|+2S_{j+1,k}^{(p)}}{\chi c},
\end{align*}
which is \eqref{eq:main_fk_bd}.
On the other hand, to show \eqref{eq:main_vk_bd}, we use
the definition of $S_{j+1,k}^{(v)}$, the fact that  $\sqrt{a+b}\leq \sqrt{a} + \sqrt{b}$ for $a,b\in \r_+$, Lemma~\ref{lem:main_descent},  and Lemma~\ref{lem:Lagr_bds}(c), to conclude that
\begin{align}
S_{j+1,k}^{(v)} & = \frac{\sum_{i=j+1}^{k}\|v^i\|}{k-j}
\leq
\left(\frac{\sum_{i=j+1}^{k}\|v^i\|^{2}}{k-j}\right)^{1/2} \nonumber \\
& \overset{\text{L.}\ref{lem:main_descent}}{\leq} \left(\frac{[\kappa_0^2 + \kappa_5 c][\Psi_j(c)-\Psi_k(c)]}{k-j}\right)^{1/2} \nonumber \\
& \overset{\text{L.}\ref{lem:Lagr_bds}(c)}{\leq} 
\sqrt{\frac{\kappa_0^2 + \kappa_5 c}{k-j}} \left[ \Delta_\phi + 4\left(\frac{\|p^{j}\|^{2}+\|p^{j-1}\|^{2}+\|p^{k}\|^{2}}{\chi^{2}c}\right) \right]^{1/2} \nonumber \\
&\leq 
2\sqrt{\frac{\kappa_0^2 + \kappa_5 c}{k-j}} \left( \Delta_\phi^{1/2} + \frac{\|p^{j}\|+\|p^{j-1}\|+\|p^{k}\|}{\chi\sqrt{c}} \right). \label{eq:Sv_bd}
\end{align}
\end{proof}
Observe that both residuals $S_{j+1,k}^{(v)}$ and $S_{j+1,k}^{(f)}$ depend on the size of the Lagrange multipliers
$p^j$, $p^{j-1}$, and $p^k$. 
If all the multipliers generated by Algorithm~\ref{alg:static_dp_admm} could be shown to be bounded independent of $c$ then it would be easy to
see that \eqref{eq:main_fk_bd}--\eqref{eq:main_vk_bd} with  $j=1$ and some $c= \Theta(\eta^{-1})$
would imply the existence of some
$k = O(\eta^{-1}\rho^{-2})$ such that
$[ S_{2,k}^{(v)}/ \rho ] + [ S_{2,k}^{(f)}/ \eta] \le 1$. 
Consequently, Algorithm~\ref{alg:static_dp_admm} would find a solution of Problem ${\cal S}_{\rho,\eta}$ in $O(\eta^{-1}\rho^{-2})$ iterations.

Unfortunately, we do not know how to bound $\{\|p_i\|\}$ independent of $c$, so we will instead show the existence
of $1\le j \le k$ such that (i) indices $j$ and $k-j$ are
$\Theta(\eta^{-1}\rho^{-2})$ 
and (ii) the three multipliers
$p^j$, $p^{j-1}$, and $p^k$ are bounded. This fact and
Proposition~\ref{prop:cross_mult_bd} suffice to show 
that the last (hypothetical) conclusion in the previous
paragraph actually holds.


\subsection{Bounding the Lagrange Multipliers}

\label{subsec:Lagr_bd}

This subsection generalizes the analysis in \cite{kong2021aidal}. More specifically, Proposition~\ref{prop:nice_bd} shows that if  $k$ is sufficiently large relative to an index $j$, the penalty parameter $c$, and $\|p^0\|$, then $S_{j+1,k}^{(p)}={\cal O}(1)$. 

The proof of the first result can be found in \cite[Lemma~B.3]{sujanani2022adaptive} using the variable substitution $(q,q^{-}, \chi)=(q^{i}, [1-\theta] p^{i-1}, c)$ and step~4 of Algorithm~\ref{alg:static_dp_admm}.

\begin{lem}\label{lem:qbounds-2} 
For every $i\geq 1$ and $r\in \pt h(z^i) + A^* q^{i}$, it holds that 
\[
\|q^i\| \leq \max \left\{(1-\theta)\|p^{i-1}\|, \frac{2 D_{\dagger} (K_h + \|r\|)}{{d}_{\ddagger} \sigma^{+}_A} \right\}.
\]
\end{lem}
The next result presents some fundamental properties about $p^{i-1}$,
$p^{i}$, and ${q}^{i}$. 
\begin{lem}
\label{lem:tech_Lagr_inexact}
For every $1\leq j \leq k$, 
\begin{itemize}
\item[(a)] $p^{j}=\chi{q}^{j}+(1-\chi)(1-\theta)p^{j-1}$; 
\item[(b)] $\|p^{j}\|\leq \|p^0\| + \kappa_1 c$;
\item[(c)] it holds that 
\[
\frac{(1-\theta)\|p^{k}\|}{k-j}+\theta S_{j+1,k}^{(p)}
\leq 
\frac{(1-\theta)\|p^j\|}{k-j} + 
\frac{2\chi D_{\dagger} \left[K_{h}+G_{f} +S_{j+1,k}^{(v)}\right]}{d_{\dagger} \sigma_{A}^{+}},
\]
where $K_h$, $d_{\dagger}$, and $(D_{\dagger}, G_f)$ are as in (A3), (A6), and \eqref{eq:agg_defs}, respectively. 
\end{itemize}
\end{lem}

\begin{proof}
(a) This is an immediate consequence of the updates for $p^{j}$ and
${q}^{j}$ in Algorithm~\ref{alg:static_dp_admm}.

(b) In view of Step~3 of Algorithm~\ref{alg:static_dp_admm}, the fact that $\theta\in(0,1)$,
and the triangle inequality, it holds that
\begin{align*}
\|p^{j}\| & \leq(1-\theta)\|p^{j-1}\|+\chi c\|f^j\| 
\leq (1-\theta)^{j}\|p^{0}\|+\chi c\sum_{i=0}^{j-1}(1-\theta)^{i}\|f^i\|\\
 & \leq\|p^{0}\|+\chi c \|A\| \sup_{z\in{\cal H}}\|z-z_{\dagger}\| \sum_{i=0}^{\infty}(1-\theta)^{i} \\
 & = \|p^{0}\|+ \frac{\chi c \|A\| D_{\dagger}}{\theta} = \|p^0\| + \kappa_1 c.
\end{align*}

(c) Let $i\geq1$ be fixed, define
\[
d_{\chi,\theta} := (1-\theta)(1-\chi),
\]
and recall that Lemma~\ref{lem:refine_props}(b) implies $v^{i}-\nabla f(x^{i})\in\partial h(x^{i})+A^{*}q^{i}$.
Using Lemma~\ref{lem:qbounds-2} with $r=v^{i}-\nabla f(x^{i})$, the
definition of $G_{f}$ in \eqref{eq:agg_defs}, and part (a), we first have that 
\begin{align*}
\|p^{i}\| & 
\overset{(a)}{=}\|\chi q^{i}+ d_{\chi,\theta} \cdot p^{i-1}\|
\leq 
\chi\|q^{i}\|+ d_{\chi,\theta} \|p^{i-1}\|\\
 & \overset{\text{L.\ref{lem:qbounds-2}}}{\leq} d_{\chi,\theta} \|p^{i-1}\|+\chi\max\left\{ (1-\theta)\|p^{i-1}\|,\frac{2D_{\dagger}(K_{h}+\|v^{i}-\nabla f(x^{i})\|)}{d_{\dagger}\sigma_{A}^{+}}\right\} \\
  & \leq (1-\theta)(1-\chi) \|p^{i-1}\| + \chi \left[(1-\theta) \|p^{i-1}\|+\frac{2D_{\dagger}(K_{h}+\|v^{i}-\nabla f(x^{i})\|)}{d_{\dagger}\sigma_{A}^{+}} \right]\\
 & \le (1-\theta)\|p^{i-1}\|+\frac{2\chi D_{\dagger}(K_{h}+\|\nabla f(x^{i})\|+\|v^{i}\|)}{d_{\dagger}\sigma_{A}^{+}}\\
 & \leq(1-\theta)\|p^{i-1}\|+\frac{2\chi D_{\dagger}(K_{h}+G_{f}+\|v^{i}\|)}{d_{\dagger}\sigma_{A}^{+}}.
\end{align*}
Summing the above inequality from $i=j+1$ to $k$ and dividing by
$k-j$ yields the desired conclusion.
\end{proof}

We are now ready to present the claimed bound on $S_{j+1,k}^{(p)}$. 

\begin{prop}
\label{prop:nice_bd} Let ${\cal R} \geq 0$ and $\underline{c}>0$ be given and suppose $c$ and $p^0$ satisfy \eqref{eq:R_cbar_bd}. Then, for any positive integers $j$ and $k$ such that $k -j \geq \kappa_6 + \xi_{\cal R}^{(0)} c + \xi_{\cal R}^{(1)}c^2$, we have
\[
S_{j+1,k}^{(p)} \leq \kappa_2,
\]
where $(\kappa_2, \kappa_6)$ and $(\xi_{\cal R}^{(0)}, \xi_{\cal R}^{(1)})$ are as in \eqref{eq:pp_indep_defs} and \eqref{eq:xi_defs}, respectively.
\end{prop}

\begin{proof} 
Using \eqref{eq:R_cbar_bd}, \eqref{eq:main_vk_bd}, Lemma~\ref{lem:tech_Lagr_inexact}(b), and the relation $\sqrt{a}+\sqrt{b} \leq \sqrt{2(a+b)}$ for $a,b\in \r_{+}$, we first have that
\begin{align*}
S_{j+1,k}^{(v)} & \leq 2\sqrt{\frac{\kappa_0^2 + \kappa_5 c}{k-j}} \left( \Delta_\phi^{1/2} + \frac{\|p^{j}\|+\|p^{j-1}\|+\|p^{k}\|}{\chi\sqrt{c}} \right) \\
& \leq \sqrt{\frac{4 (\kappa_0^2 + \kappa_5 c)}{k-j}} \left( \Delta_\phi^{1/2} + \frac{3[\|p^{0}\|+\kappa_{1}c]}{\chi\sqrt{c}} \right) \\
& \leq \sqrt{\frac{4 (\kappa_0^2 + \kappa_5 c)}{k-j}} \left( \Delta_\phi^{1/2} + \frac{3[{\cal R}+\kappa_{1}]\sqrt{c}}{\chi} \right) \\ 
& \le
\sqrt{\frac{8 (\kappa_0^2 + \kappa_5 c)}{k-j} \left( \Delta_\phi + \frac{9[{\cal R}+\kappa_{1}]^2 c} {\chi^2} \right) } 
\leq  \kappa_4 \sqrt{\frac{\xi_{\cal R}^{(0)}c + \xi_{\cal R}^{(1)}c^2}{k-j}}.
\end{align*}
Using the above bound, Lemma~\ref{lem:tech_Lagr_inexact}(b)--(c), our assumed bound on $k-j$, and the definition of $\kappa_2$, we conclude that
\begin{align*}
S_{j+1,k}^{(p)} & 
\leq \frac{2\chi D_{\dagger}(K_{h}+G_{f})}{\theta d_{\dagger}\sigma_{A}^{+}} + \frac{(1-\theta)\|p^{j}\|}{\theta(k-j)}+ \frac{S_{j+1,k}^{(v)}}{\kappa_4}  \\
 & \leq \frac{2\chi D_{\dagger}(K_{h}+G_{f})}{\theta d_{\dagger}\sigma_{A}^{+}} + \frac{(1-\theta)(\|p^{0}\|+\kappa_{1}c)}{\theta(k-j)}+ \sqrt{\frac{\kappa_6 + \xi_{\cal R}^{(0)} c + \xi_{\cal R}^{(1)} c^2}{k-j}}\\
& \leq \frac{2\chi D_{\dagger}(K_{h}+G_{f})}{\theta d_{\dagger}\sigma_{A}^{+}} + \frac{(1-\theta)({\cal R} + \kappa_1)c}{\theta(k-j)}+\sqrt{\frac{\kappa_6 + \xi_{\cal R}^{(0)} c + \xi_{\cal R}^{(1)} c^2}{k-j}}\\
& \leq  \frac{2\chi D_{\dagger}(K_{h}+G_{f})}{\theta d_{\dagger}\sigma_{A}^{+}} + \frac{\xi_{\cal R}^{(0)} c}{\theta(k-j)} + \sqrt{\frac{\kappa_6 + \xi_{\cal R}^{(0)} c + \xi_{\cal R}^{(1)} c^2}{k-j}} \\
& \leq \frac{1}{\theta}\left[1 + \frac{2\chi D_{\dagger}(K_{h}+G_{f})}{\theta d_{\dagger}\sigma_{A}^{+}}\right] + 1 = \kappa_2.
\end{align*}
\end{proof}

We end this subsection by discussing some implications of the above results.
Suppose $\zeta$ is an integer satisfying $\zeta \geq \kappa_6 + \xi_{\cal R}^{(0)}c + \xi_{\cal R}^{(1)}c^2 = \Theta(c^2)$.
It then follows from Proposition~\ref{prop:nice_bd}
that $S_{2,\zeta}^{(p)} = {\cal O}(1)$ and $S_{2\zeta,3\zeta}^{(p)} = {\cal O}(1)$. Since the minimum
of a set of scalars minorizes its average, there exist indices 
$j_0\in\{2,\ldots,\zeta\}$ and $k_0\in\{2\zeta, \ldots, 3\zeta\}$ such that $\|p^{j_0}\|={\cal O}(1)$
and $\|p^{k_0}\|={\cal O}(1)$. Using the fact that $k_0 - j_0 \geq \zeta$, 
the above bounds, and \eqref{eq:main_fk_bd}--\eqref{eq:main_vk_bd} with $(j,k)=(j_0,k_0)$, it is reasonable to expect that
$S_{j_0+1,k_0}^{(f)}={\cal O}(1/c)$ and $S_{j_0+1,k_0}^{(v)}={\cal O}(\sqrt{c/\zeta})$.
In the next section, we give the exact steps of this argument and use the resulting bounds to
prove Proposition~\ref{prop:cycle_props}.

\subsection{Proof of Proposition~\ref{prop:cycle_props}}

\label{sec:prop_prf}

Before presenting the proof
of Proposition~\ref{prop:cycle_props}, we first give two technical results. The first one refines the bounds in Proposition~\ref{prop:cross_mult_bd} using Proposition~\ref{prop:nice_bd}, while the second one gives an important implication of \eqref{eq:penalty_thresh}. 

\begin{lem}
\label{lem:spec_term_bd} Let ${\cal R} \geq 0$ and $\underline{c}>0$ be given and suppose $(c,p^0)$ satisfies \eqref{eq:R_cbar_bd} for some $\cal R\geq 0$ and $\underline{c}>0$. For any integer $\zeta$ such that $\zeta \geq \kappa_6 + \xi_{\cal R}^{(0)} c + \xi_{\cal R}^{(1)}c^2$, there exist $j \in \{3,\ldots, \zeta\}$ and $k \in \{2\zeta+1, \ldots, 3 \zeta\}$ satisfying
\begin{align}
S_{j+1,k}^{(v)}
\leq
\tilde{\kappa}_{\underline{c}}^{(0)} \sqrt{\frac{\kappa_0^2 + \kappa_5 c}{k - j}},
\quad
 S_{j+1,k}^{(f)}
\leq
\frac{6\kappa_{2}}{\chi c}, \label{eq:refined_S_bds}
\end{align}
where $(\kappa_0, \kappa_2, \kappa_5)$ and $\tilde{\kappa}_0$ is are as in \eqref{eq:pp_indep_defs} and \eqref{eq:tilde_kappa_defs}, respectively.
\end{lem}

\begin{proof}
Suppose $\zeta\in \mathbb{N}$ satisfies $\zeta \geq \kappa_6 + \xi_{\cal R}^{(0)} c + \xi_{\cal R}^{(1)}c^2$. Using  Proposition~\ref{prop:nice_bd}
with $(j,k)=(1,\zeta)$ it holds that there exists $3\leq j\le \zeta$
such that 
\begin{align}
\|p^{j-1}\|+\|p^{j}\| & 
\leq
\frac{\sum_{i=3}^{\zeta}(\|p^{i-1}\|+\|p^{i}\|)}{\zeta-2} \nonumber \leq
\frac{2\sum_{i=2}^{\zeta}\|p^{i}\|}{\zeta-2} \\
& =
 \frac{2(\zeta-1)S_{2,\zeta}^{(p)}}{\zeta-2}\leq4S_{2,\zeta}^{(p)}\leq 4 \kappa_2.
 \label{eq:pkl_bd}
\end{align}
On the other hand, using Proposition~\ref{prop:nice_bd} with $(j,k)=(2\zeta,3\zeta)$ it
holds that there exists $k\in \{2\zeta +1,\ldots,3\zeta\}$ such that 
\begin{equation}
\|p^{k}\|\leq\frac{\sum_{i=2\zeta+1}^{3\zeta}\|p^{i}\|}{\zeta}=S_{2\zeta+1,3\zeta}
\leq 
\kappa_2 .\label{eq:pku_bd}
\end{equation}
Combining \eqref{eq:pkl_bd}, \eqref{eq:pku_bd}, and Proposition~\ref{prop:cross_mult_bd}, it
follows that 
\begin{align*}
S_{j+1,k}^{(v)} & 
\leq 2\sqrt{\frac{\kappa_0^2 + \kappa_5 c}{k - j}} \left(\Delta_{\phi}^{1/2}+\frac{\|p^{j_{0}}\|+\|p^{j_{0}-1}\|+\|p^{k_{0}}\|}{\chi\sqrt{c}}\right) \\
& \overset{\eqref{eq:pkl_bd}\text{--}\eqref{eq:pku_bd}}{\leq} 2\sqrt{\frac{\kappa_0^2 + \kappa_5 c}{k - j}} \left(\Delta_{\phi}^{1/2}+\frac{5 \kappa_2}{\chi\sqrt{c}}\right) \\
& \leq 
2{\sqrt{\frac{\kappa_0^2 + \kappa_5 c}{k - j}}}\left(\Delta_{\phi}^{1/2}+\frac{5\kappa_2}{\chi\sqrt{\underline{c}}}\right)
=
\tilde{\kappa}_{\underline{c}}^{(0)} \sqrt{\frac{\kappa_0^2 + \kappa_5 c}{k - j}},
\end{align*}
which is the first bound in \eqref{eq:refined_S_bds}. To show the other bound in \eqref{eq:refined_S_bds}, we use \eqref{eq:pkl_bd} and Proposition~\ref{prop:nice_bd} to conclude that 
\[
S_{j+1,k}^{(f)}\leq\frac{\|p^{j}\|+2S_{j+1,k}^{(p)}}{\chi c} \leq
\frac{6\kappa_2}{\chi c}.
\]
\end{proof}

We now state a technical result which will be used in the proof
of
Proposition~\ref{prop:cycle_props}(c).

\begin{lem} \label{lem:penalty_tech_res} 
For any ${\cal R} \geq 0$ and $c \ge \underline{c}>0$, the following statements hold:
\begin{itemize}
    \item[(a)] the quantity  ${\cal T}_c(\cdot,\cdot\,|\,\cdot,\cdot)$ 
    defined in \eqref{eq:T_static_def}
    satisfies
    \[
    {\cal T}_c(\rho,\eta\,|\,\underline{c},{\cal R}) \leq \left[\left(\frac{c}{\underline{c}}\right)^2 + \frac{c}{\underline{c} \cdot \min\{\rho^2,\eta^2\}}\right] {\cal T}_{\underline{c}}(1,1\,|\,\underline{c},{\cal R});
    \]
    \item[(b)] if $c$ satisfies \eqref{eq:penalty_thresh}, then 
${\cal T}_c(\rho,\eta\,|\,\underline{c},{\cal R}) \leq c^3$.
\end{itemize}
\end{lem}

\begin{proof}

(a) This statement follows immediately from the definition of ${\cal T}_c(\cdot,\cdot\,|\,\cdot,\cdot)$ and
the fact that for any $c\geq \bar{c}$
any nonnegative scalars $\alpha$, $\beta$, and $\gamma$, we have
\[
\alpha + \beta c \leq (\alpha + \beta \underline{c}) \left( \frac{c}{\underline{c}} \right), \quad
\alpha + \beta c + \gamma c^2 \leq (\alpha + \beta \underline{c} + \gamma \underline{c}^2) 
\left(\frac{c}{\underline{c}}\right)^2.
\]

(b)  Define  $\hat{c} := \hat{c}(\rho,\eta\,|\,\underline{c},{\cal R})$, $\varepsilon := \min\{\rho,\eta\}$, and $T := {\cal T}_{\underline{c}}(1,1\,|\,\underline{c},{\cal R})$,
and assume that $c$ satisfies \eqref{eq:penalty_thresh}, or equivalently,
$c \ge \hat c$. To show the conclusion of (b), it suffices to show that 
\begin{equation}
\left[\left(\frac{c}{\underline{c}}\right)^2 + \frac{c}{\underline{c} \cdot \varepsilon^2}\right] T  \leq c^3. \label{eq:suff_penalty_thresh}    
\end{equation}
in view of part (a).
It is easy to see that the above inequality is satisfied by any $c$ such that
\[
c \ge \pi_{\varepsilon} := \frac{ T/\underline{c}^2  + \sqrt{T^2/\underline{c}^4 + 4T /(\varepsilon^2\underline{c})}}{2}.
\] 
Since the definition of
$\hat{c}$ in \eqref{eq:penalty_thresh} and 
the relation $\sqrt{a+b} \leq \sqrt{a}+\sqrt{b}$ for $a,b\in\r_+$
imply that
$\hat{c} \geq \pi_\varepsilon$,
the conclusion of (b) follows from 
the assumption
that $c\geq \hat{c}$ and
the
previous observation.
\end{proof}

We now remark on
Lemma~\ref{lem:spec_term_bd}. For any integer
$\zeta \ge \kappa_6 + \xi_{\cal R}^{(0)} c + \xi_{\cal R}^{(1)}c^2$, it follows that
there exist $i_1,i_2 \le 3 \zeta$
such that $\|v_{i_1}\| = {\cal O}(\sqrt{c/\zeta)})$ and
$\|f_{i_2}\| = {\cal O}(1/c)$.
Hence, for some 
$c=\Theta(\eta^{-1})$ and some
$\zeta \ge \Omega(\rho^{-2}\eta^{-1})$,
we can guarantee that
$\|v_{i_1}\| \le \rho $ and
$\|f_{i_2}\| \le \eta$. Clearly,
if $i_1=i_2$ then this argument
shows that a solution of Problem~${\cal S}_{\rho,\eta}$ can be
found in ${\cal O}(\rho^{-2}\eta^{-1})$ iterations of Algorithm~\ref{alg:static_dp_admm}.
In the proof (of Proposition~\ref{prop:cycle_props}) below, we give a more involved argument that guarantees that the above
$i_1$ and $i_2$ can be chosen so
that $i_1=i_2$.

\begin{proof}[Proof of Proposition~\ref{prop:cycle_props}]
(a) Let $(\rho, \eta)\in \r_{++}^2$, $p^0\in A(\rn)$, and $c>0$ be given, and define
\[
T := {\cal T}_c(\rho,\eta\,|\,\underline{c},{\cal R}), \quad r_j := \frac{{\cal S}_{j}^{(v)}}{\rho} + \frac{{\cal S}_{j}^{(f)}}{\eta} \sqrt{\frac{c^3}{j}} \quad  \forall j\geq 1,
\]
where ${\cal S}_j^{(v)}$ and ${\cal S}_j^{(f)}$ are as in Step~2b of Algorithm~\ref{alg:static_dp_admm} and ${\cal T}_c(\cdot,\cdot\,|\,\cdot,\cdot)$ is as in \eqref{eq:T_static_def}. 
For the sake of contradiction, suppose
that Algorithm~\ref{alg:static_dp_admm} has not terminated by the end of iteration $k = T$.
Since Algorithm~\ref{alg:static_dp_admm} (see its Step~2b) terminates unsuccessfully at iteration $k$ exactly when $r_k \leq 1$, we will obtain the
desired contradiction by showing that
 there exists $k\leq T$ such that $r_k \leq 1$.

First, consider an arbitrary
pair of integers $j$ and $k$ 
such that $1 \le j \le k \le T$ and
assume without loss of generality that $k$ is even.
Then, 
combining \eqref{eq:tech_main_prop_bd}, the relations $S_{k/2,k}^{(v)} = {\cal S}_{k}^{(v)}$ and $S_{k/2,k}^{(f)} = {\cal S}_{k}^{(f)}$, we easily see that
\begin{align}
r_{k} & =\frac{S_{k/2,k}^{(v)}}{\rho}+\frac{c^{3/2}  S_{k/2,k}^{(f)}}{\eta\sqrt{k}}
 =
\frac{k-j+1}{k-k/2+1}\left[\frac{S_{j,k}^{(v)}}{\rho}+\frac{c^{3/2} 
S_{j,k}^{(f)}}{\eta\sqrt{k}}\right] \nonumber \\
& \leq
\frac{k+2}{k/2+1}\left[\frac{S_{j,k}^{(v)}}{\rho}+\frac{c^{3/2} 
S_{j,k}^{(f)}}{\eta\sqrt{k}}\right] 
=
2\left[\frac{S_{j,k}^{(v)}}{\rho}+\frac{c^{3/2} 
S_{j,k}^{(f)}}{\eta\sqrt{k}}\right],
\label{eq:rk0_bd}
\end{align}
We now show that there exists suitable $j$ and $k$ so that the last expression is bounded by 1 and 
hence that our desired contradiction follows.
Note first that the definition of $T={\cal T}_c(\rho,\eta)$ in
\eqref{eq:T_static_def} implies that 
$\zeta := T/3$ satisfies
the assumption of 
Lemma~\ref{lem:spec_term_bd}. Hence,
the conclusion of this lemma implies
the existence of $j\in\{3, \ldots, T/3\}$ and $k \in \{2T/3 + 1, \ldots, T \}$ such that
\begin{align}
\frac{S_{j,k}^{(v)}}{\rho} + 
\frac{c^{3/2} S_{j,k}^{(f)}}{\eta\sqrt{k}} & 
\leq
\frac{\tilde{\kappa}_{\underline{c}}^{(0)} \sqrt{\kappa_0^2 + \kappa_5 c}}{\rho\sqrt{k - j}}+
\frac{6\kappa_{2}\sqrt{c}}{\chi\eta\sqrt{k}} 
\leq
\frac{\tilde{\kappa}_{\underline{c}}^{(0)} \sqrt{\kappa_0^2 + \kappa_5 c}}{\rho\sqrt{T/3}}+
\frac{6\kappa_{2}\sqrt{c}}{\chi\eta\sqrt{T/3}} \nonumber\\
& =\sqrt{\frac{\tilde{\kappa}_{1}+\tilde{\kappa}_2 c}{\rho^{2}T}}+
\sqrt{\frac{\kappa_{3} c}{\eta^{2}T}} 
\leq\frac{1}{4}+\frac{1}{4}=\frac{1}{2}, \label{eq:tech_main_prop_bd}
\end{align}
where the last inequality follows from the definition of $T$.
Combining \eqref{eq:rk0_bd} and \eqref{eq:tech_main_prop_bd} we conclude that $r_k\le 1$, which yields our desired contradiction.

(b) This follows immediately from the stopping condition in Step~2a of Algorithm~\ref{alg:static_dp_admm}
and Lemma~\ref{lem:refine_props}(b).

(c) Let $(T,r_k)$ be as in part (a) and assume that $c$ satisfies \eqref{eq:penalty_thresh}.  Assume, for contradiction, that
Algorithm~\ref{alg:static_dp_admm} does not terminate successfully. Then, by part (a), 
the algorithm terminates in an iteration $k \leq T$ such that $r_{k} \leq 1$. Using the fact that $r_k$ itself is an average of scalars, there exists $k/2 \leq i\leq k$ such that 
\[
\frac{\|v^i\|}{\rho} + \frac{c^{3/2} \|f^i\|}{\eta \sqrt{k}} 
\leq 
\frac{{S}_{k/2,k}^{(v)}}{\rho} + \frac{c^{3/2} {S}_{k/2,k}^{(f)}}{\eta \sqrt{k}} 
\leq 1.
\]
Hence, it holds that $\|v^i\|\leq\rho$
and $\|f^i\| \leq \eta\sqrt{k} c^{-3/2} \leq  \eta\sqrt{T} c^{-3/2}$ where the last inequality is due to the fact that $k \le T$. 
Moreover, the assumption that $c$
satisfies \eqref{eq:penalty_thresh} together with Lemma~\ref{lem:penalty_tech_res}(b) then imply that $T\leq c^3$ and, hence, that $\|f^i\| \leq \eta$. Consequently, this means that the algorithm actually terminates
successfully at iteration $i \le k$.
We have thus established the desired contradiction and, hence, that
part (c) holds.
\end{proof}

\section{Analysis of Algorithm~\ref{alg:dp_admm}} 

\label{sec:dynamic_analysis}

This section presents the main properties of Algorithm~\ref{alg:dp_admm}, including the proof of Theorem~\ref{thm:total_compl}.

We first start with two crucial technical results. 

\begin{prop}
\label{prop:penalty_thresh}
The following statements hold about the $\ell^{\rm th}$ iteration of Algorithm~\ref{alg:dp_admm}:
\begin{itemize}
    \item[(a)]$\|\bar{p}^{\ell-1}\| / c_{\ell} \leq 2\kappa_{1}$, where $\kappa_1$ is as in \eqref{eq:pp_indep_defs};
    \item[(b)] its call to Algorithm~\ref{alg:static_dp_admm} terminates in
    ${\cal{T}}_{c_{\ell}}(\rho,\eta\,|\,c_1, 2\kappa_1)$
    iterations and,
    if the $\ell^{\rm th}$ penalty parameter $c_\ell>0$ satisfies 
\begin{equation}
    c_{\ell} \geq \hat{c}(\rho,\eta\,|\,c_1,2\kappa_1),
\end{equation}
then this call
terminates successfully, where $\kappa_1$,  ${\cal T}_{c}(\cdot,\cdot\,|\,\cdot,\cdot)$, and $\hat{c}(\cdot, \cdot\,|\,\cdot,\cdot)$ are as in \eqref{eq:pp_indep_defs},  \eqref{eq:T_static_def}, and \eqref{eq:penalty_thresh}, respectively.
\end{itemize}
\end{prop}

\begin{proof}
(a) We proceed by induction. Since $\bar{p}^0=0$, the case of $\ell=1$ is immediate. Suppose the statement holds for some iteration $\ell$ and, hence, that
$\|\bar p^{\ell-1}\| \leq 2 \kappa_1 c_\ell$. Then, it follows from Lemma~\ref{lem:tech_Lagr_inexact}(b) with $(p^0,c)=(\bar{p}^{\ell-1},c_\ell)$ and the relation $c_{\ell+1} = 2c_{\ell}$ that 
\[
\|\bar{p}^{\ell}\| \leq \|\bar{p}^{\ell-1}\| + \kappa_{1} c_{\ell} \leq 2 \kappa_{1} c_{\ell} + \kappa_1 c_{\ell} = 3 \kappa_1 c_{\ell} = \frac{3 \kappa_1}{2}  c_{\ell+1} < 2 \kappa_1 c_{\ell+1}.
\]

(b) This follows from part (a), the fact that $\{c_\ell\}_{\ell\geq 1}$ is an increasing sequence,  and Proposition~\ref{prop:cycle_props} with $(c, \underline{c}, {\cal R})=(c_\ell, c_1, 2\kappa_1)$.
\end{proof}

We are now ready to give the proof of Theorem~\ref{thm:total_compl}.

\begin{proof}[Proof of Theorem~\ref{thm:total_compl}]
Define the scalars
\begin{gather*}
\hat{c} := \hat{c}(\rho,\eta\,|\,c_1, 2\kappa_1), \quad \hat{\ell} := \lceil \log_{2}^{+}(\hat{c} / c_1) \rceil, 
\quad {\cal T}_{c_\ell} := {\cal T}_{c_\ell}(\rho,\eta\,|\,c_1, 2\kappa_1),
\end{gather*}
where $\hat{c}(\cdot,\cdot\,|\,\cdot,\cdot)$ is as in \eqref{eq:penalty_thresh}.
Proposition~\ref{prop:penalty_thresh}(b) and the update rule for $c_\ell$
imply that Algorithm~\ref{alg:dp_admm} performs at most $\hat{\ell}$ iterations, and terminates with a pair that solves Problem~${\cal S}_{\rho,\eta}$.
Moreover, the total number 
of iterations of Algorithm~\ref{alg:static_dp_admm}
(performed by all of Algorithm~\ref{alg:dp_admm}'s calls to it) is
bounded by $\sum_{\ell=1}^{\hat{\ell}}{\cal T}_{c_\ell}$. Now, using Lemma~\ref{lem:penalty_tech_res}(a) with $\underline{c}=c_1$, it follows that
\begin{equation}
\frac{\sum_{\ell=1}^{\hat{\ell}}{\cal T}_{c_\ell}}{T_1} \leq \frac{\sum_{\ell=1}^{\hat{\ell}}c_{\ell}^{2}}{c_{1}^{2}}+\frac{\sum_{\ell=1}^{\hat{\ell}}c_{\ell}}{c_{1}\varepsilon^2} = \sum_{\ell=1}^{\hat{\ell}}2^{2(\ell-1)}+\frac{\sum_{\ell=1}^{\hat{\ell}}2^{(\ell-1)}}{\varepsilon^2} \leq 4^{\hat{\ell}}+\frac{2^{\hat{\ell}}}{\varepsilon^{2}},
\label{eq:T_hat_bd}
\end{equation}
where $(T_1,\varepsilon)$ are as in \eqref{eq:dyn_aliases}.
We now derive suitable bounds for $4^{\hat\ell}$ and $2^{\hat\ell}$. Using the definitions of $\hat{c}$ and $\hat{\ell}$, and the definition of $(E_0,E_1)$ in \eqref{eq:E_defs}, we first have that
\begin{align}
2^{\hat{\ell}} & \leq\max\left\{ 2,2^{(1+\log_{2}\hat{c}/c_{1})}\right\} \leq2\max\left\{ 1,\frac{\hat{c}}{c_{1}}\right\} = 2\max\left\{ 1,\frac{1}{c_{1}^{3}}\left(T_1+\frac{\sqrt{c_1^3 T_1}}{\varepsilon}\right)\right\} 
\nonumber \\
& \leq 2\left(1+\frac{T_1}{c_{1}^{3}}+\frac{1}{\varepsilon}\sqrt{\frac{T_1}{c_1^{3}}}\right) 
= E_0 +\frac{E_1}{\varepsilon} \label{eq:2l_bd}.
\end{align}
Combining the above inequality above with the bound $(a+b)^2 \leq 2a^2 + 2b^2$ for $a,b\in\r$, it is also easy to see that
\begin{align}
4^{\hat{\ell}} & \leq (2^{\hat \ell})^2 \leq  2 E_0^2 + \frac{2E_1^2}{\varepsilon^2}. \label{eq:4l_bd}
\end{align}
The conclusion now follows by applying \eqref{eq:4l_bd} and \eqref{eq:2l_bd} to \eqref{eq:T_hat_bd}.
\end{proof}

\section{Numerical Experiments}

\label{sec:numerical_experiments}
This section examines the performance of the proposed DP.ADMM (Algorithm~\ref{alg:dp_admm}) for finding stationary points of a nonconvex three-block distributed quadratic programming problem. Specifically, given a radius $\gamma > 0$ and a dimension $n\in \n$, it considers the three-block problem
\begin{align*}
\min_{(x_{1},x_{2},x_{3})\in\rn\times\rn\times\rn}\  & -\sum_{i=1}^{2}\left[\frac{\alpha_{i}}{2}\|x_{i}\|^{2}+\left\langle x_{i},\beta_{i}\right\rangle \right]\\
\text{s.t.}\  & \|x\|_{\infty}\leq\gamma,\\
 & x_{1}-x_{3}=0,\\
 & x_{2}-x_{3}=0,
\end{align*}
where $\{\alpha_i\}_{i=1}^2\subseteq [0,1]$, $\{\beta_i\}_{i=1}^2\subseteq [0,1]^n$, and the entries of these quantities are sampled from the uniform distribution on $[0,1]$. It is clear that the above problem is an instance of \eqref{eq:main_prb} if we take $h_i$ to be the indicator of the set $\{x\in\rn:\|x\|_\infty \leq \gamma\}$ for $i=1,\ldots,3$.
At the end of this section, we give some elucidating remarks.

Before presenting the results, we first describe the algorithms tested. 
The first set of algorithms, labeled DP1--DP2, are modifications of Algorithm~\ref{alg:dp_admm}. \
Specifically, both DP1 and DP2 replace the original definition of ${\cal S}_k^{(f)}$ (resp. ${\cal S}_k^{(f)}$) in Step~2b of Algorithm~\ref{alg:static_dp_admm} with $2\sum_{i=1}^k \|v^i\|/[k+2]$ (resp. $2\sum_{i=1}^k \|Ax^i-d\|/[k+2]$) and choose $(\lam,c_1)=(1/2,1)$. 
Moreover, DP1 chooses $(\theta,\chi)=(0,1)$ while DP2 chooses $(\theta,\chi)=(1/2,1/18)$ which satisfies \eqref{eq:chi_theta_cond} at equality. The second set of algorithms, labeled SDD1--SDD3, are instances of the SDD-ADMM of \cite{sun2021dual} for different values of the penalty parameter $\rho$. Specifically, all of these instances uses the parameters $(\omega,\theta,\tau)=(4,2,1)$, following the same choice as in \cite[Section~5.1]{sun2021dual}, and select the following curvature constants: $(M_h, K_h, J_h, L_h)=(4\gamma, 1, 1, 0)$. Moreover, SDD1--SDD3 respectively choose the penalty parameter $\rho$ to be $0.1$, $1.0$, and $10.0$, and termination of the method occurs when the norm of the stationary residual $\xi^k$ and feasibility are both less than a given numerical tolerance.

The results of our experiment are now given in Tables~\ref{tab:gamma}--\ref{tab:n}, which present both iteration counts and runtimes for either varying choices of $\gamma$ (Table~\ref{tab:gamma}) or $n$ (Table~\ref{tab:n}). We now describe a few more details about these experiments and tables. 
First, the starting point for all methods is the zero vector and the numerical tolerances (e.g., $\rho$ and $\eta$ in DP1--DP2) for each method were set to be $10^{-9}$. 
Second, the bolded text in the tables highlight the method that performed the best in terms of iteration count. 
Third, we imposed an iteration limit of 100,000 and marked the runs which did not terminate by this limit with a `-' symbol. Fourth, the experiments were implemented and executed in Matlab R2021b on a Windows 64-bit desktop machine with 12GB of RAM and two Intel(R) Xeon(R) Gold 6240 processors, and the code is readily available online\footnote{See \url{https://github.com/wwkong/nc_opt/tree/master/tests/papers/dp_admm}.}.

\begin{table}[!tbh]
\begin{centering}
{\tiny{}}%
\begin{tabular}{c|ccccc|ccccc}
 & \multicolumn{5}{c|}{{\footnotesize{}Iteration Count}} & \multicolumn{5}{c}{{\footnotesize{}Runtime (ms)}}\tabularnewline
{\footnotesize{}$\gamma$} & {\footnotesize{}DP1} & {\footnotesize{}DP2} & {\footnotesize{}SDD1} & {\footnotesize{}SDD2} & {\footnotesize{}SDD3} & {\footnotesize{}DP1} & {\footnotesize{}DP2} & {\footnotesize{}SDD1} & {\footnotesize{}SDD2} & {\footnotesize{}SDD3}\tabularnewline
\hline 
{\footnotesize{}$10^{0}$} & \textbf{\footnotesize{}21} & {\footnotesize{}29} & {\footnotesize{}363} & {\footnotesize{}135} & {\footnotesize{}528} & {\footnotesize{}1.8} & {\footnotesize{}1.9} & {\footnotesize{}38.2} & {\footnotesize{}13.4} & {\footnotesize{}50.4}\tabularnewline
{\footnotesize{}$10^{1}$} & \textbf{\footnotesize{}76} & {\footnotesize{}83} & {\footnotesize{}427} & {\footnotesize{}223} & {\footnotesize{}976} & {\footnotesize{}4.0} & {\footnotesize{}4.9} & {\footnotesize{}41.3} & {\footnotesize{}22.4} & {\footnotesize{}88.1}\tabularnewline
{\footnotesize{}$10^{2}$} & \textbf{\footnotesize{}151} & {\footnotesize{}156} & {\footnotesize{}497} & {\footnotesize{}309} & {\footnotesize{}1394} & {\footnotesize{}7.9} & {\footnotesize{}7.7} & {\footnotesize{}45.2} & {\footnotesize{}28.3} & {\footnotesize{}121.7}\tabularnewline
{\footnotesize{}$10^{3}$} & \textbf{\footnotesize{}228} & {\footnotesize{}232} & {\footnotesize{}569} & {\footnotesize{}399} & {\footnotesize{}1855} & {\footnotesize{}10.8} & {\footnotesize{}10.8} & {\footnotesize{}51.2} & {\footnotesize{}34.3} & {\footnotesize{}159.3}\tabularnewline
{\footnotesize{}$10^{4}$} & \textbf{\footnotesize{}306} & {\footnotesize{}308} & {\footnotesize{}647} & {\footnotesize{}489} & {\footnotesize{}2316} & {\footnotesize{}15.5} & {\footnotesize{}17.6} & {\footnotesize{}58.9} & {\footnotesize{}42.9} & {\footnotesize{}223.1}\tabularnewline
{\footnotesize{}$10^{5}$} & \textbf{\footnotesize{}385} & {\footnotesize{}385} & {\footnotesize{}-} & {\footnotesize{}581} & {\footnotesize{}2778} & {\footnotesize{}17.9} & {\footnotesize{}18.5} & {\footnotesize{}-} & {\footnotesize{}48.0} & {\footnotesize{}241.5}\tabularnewline
\end{tabular}{\tiny\par}
\par\end{centering}
\caption{Results with $n=10$ and different values of $\gamma$\label{tab:gamma}}
\end{table}

\begin{table}[!tbh]
\begin{centering}
{\tiny{}}%
\begin{tabular}{c|ccccc|ccccc}
 & \multicolumn{5}{c|}{{\footnotesize{}Iteration Count}} & \multicolumn{5}{c}{{\footnotesize{}Runtime (ms)}}\tabularnewline
{\footnotesize{}$n$} & {\footnotesize{}DP1} & {\footnotesize{}DP2} & {\footnotesize{}SDD1} & {\footnotesize{}SDD2} & {\footnotesize{}SDD3} & {\footnotesize{}DP1} & {\footnotesize{}DP2} & {\footnotesize{}SDD1} & {\footnotesize{}SDD2} & {\footnotesize{}SDD3}\tabularnewline
\hline 
{\footnotesize{}10} & \textbf{\footnotesize{}151} & {\footnotesize{}156} & {\footnotesize{}497} & {\footnotesize{}309} & {\footnotesize{}1394} & {\footnotesize{}7.8} & {\footnotesize{}7.5} & {\footnotesize{}65.8} & {\footnotesize{}29.0} & {\footnotesize{}121.8}\tabularnewline
{\footnotesize{}40} & \textbf{\footnotesize{}55} & {\footnotesize{}60} & {\footnotesize{}-} & {\footnotesize{}-} & {\footnotesize{}3117} & {\footnotesize{}3.7} & {\footnotesize{}3.5} & {\footnotesize{}-} & {\footnotesize{}-} & {\footnotesize{}319.0}\tabularnewline
{\footnotesize{}160} & \textbf{\footnotesize{}139} & {\footnotesize{}144} & {\footnotesize{}-} & {\footnotesize{}388} & {\footnotesize{}1836} & {\footnotesize{}8.5} & {\footnotesize{}8.2} & {\footnotesize{}-} & {\footnotesize{}42.0} & {\footnotesize{}202.7}\tabularnewline
{\footnotesize{}640} & \textbf{\footnotesize{}53} & {\footnotesize{}54} & {\footnotesize{}-} & {\footnotesize{}349} & {\footnotesize{}16243} & {\footnotesize{}4.0} & {\footnotesize{}3.9} & {\footnotesize{}-} & {\footnotesize{}40.4} & {\footnotesize{}1901.5}\tabularnewline
{\footnotesize{}2560} & \textbf{\footnotesize{}58} & {\footnotesize{}59} & {\footnotesize{}-} & {\footnotesize{}458} & {\footnotesize{}8464} & {\footnotesize{}7.1} & {\footnotesize{}6.7} & {\footnotesize{}-} & {\footnotesize{}77.4} & {\footnotesize{}1553.7}\tabularnewline
{\footnotesize{}10240} & \textbf{\footnotesize{}108} & {\footnotesize{}110} & {\footnotesize{}-} & {\footnotesize{}1058} & {\footnotesize{}4334} & {\footnotesize{}44.4} & {\footnotesize{}40.3} & {\footnotesize{}-} & {\footnotesize{}623.5} & {\footnotesize{}2790.6}\tabularnewline
\end{tabular}{\tiny\par}
\par\end{centering}
\caption{Results with $\gamma=100$ and different values of $n$\label{tab:n}}
\end{table}

From the results in Tables~\ref{tab:gamma}--\ref{tab:n}, we see that DP1 performed the best in terms of iteration count and DP2 had iteration counts that were close to DP1. 
On the other hand, SDD2 outperformed its other SDD-ADMM variant on all problems except one. 
Finally, notice that the DP.ADMM variants scaled better against the dimension $n$ compared to the SDD-ADMM variants.

To close this section, we give some elucidating remarks. 
First, we excluded the algorithm in \cite{jiang2019structured} due to its poor iteration complexity bound and the fact that it is an algorithm applied to a reformulation of \eqref{eq:main_prb} rather than to \eqref{eq:main_prb} directly. 
Second, we had to choose different values of the penalty parameter $\rho$ for the SDD-ADMM variants because the analysis in \cite{sun2021dual} did not present a practical way of adaptively updating $\rho$ (note that the ``adaptive'' method in \cite[Algorithm~3.2]{sun2021dual} is not practical because it requires an estimate of $\sup_{x\in {\cal H}} \phi(x) - \inf_{x\in{\cal H}} \phi$ for \eqref{eq:main_prb}).

\section{Concluding Remarks}

\label{sec:concluding_remarks}

The analysis of this paper also applies to instances of \eqref{eq:main_prb} where
$f$ is not necessarily differentiable on
${\cal H}$ as in our condition (A5), but instead satisfies a more relaxed version of (A5), namely:
for every $x \in {\cal H}$, the function
$f(x_{<t},\cdot,x_{>t})$
has a Fr\'echet subgradient at $x_t$,
denoted by
$\nabla_{x_{t}}f(x_{\leq t},{x}_{>t})$, and
\eqref{eq:lipschitz_x}
is satisfied for every $t=1,\ldots,B-1$.
Hence, our analysis immediately applies to
the case where $f(z)$ is of the form $\sum_{t=1}^B f_t(z_t)$
in which, for every $t=1,\ldots,B$, the function
$f_t(\cdot) +m_t\|\cdot\|^2/2+\delta_{{\cal H}_t}(\cdot)$ is convex and has a subgradient
everywhere in ${\cal H}_t$.

We now discuss some possible extensions of our analysis in this paper. 
First, our analysis was done
under the assumption that ${\cal H}$ is
bounded (see (A3)), but
it is straightforward to see that it is still
valid under the weaker assumption
that $\sup_{k \geq 1} \|x^k - z_\dagger \| \le D_\dagger$
for some $D_\dagger >0$
where $z_\dagger$ is as in (A6).
It would be interesting to extend the analysis in this paper to the case where ${\cal H}$ is unbounded, possibly by assuming conditions on
the sublevel sets of $\phi$
which guarantee that the aforementioned bound holds.
Second, the convergence of Algorithm~\ref{alg:dp_admm} is established under the assumption that exact solutions to the subproblems in Step~1 of Algorithm~\ref{alg:static_dp_admm} are easy to obtain. 
We believe that convergence can also be established when only inexact solutions, e.g.,
\begin{equation}
x_{t}^{k} \approx \argmin_{u_t\in\r^{n_{t}}}\left\{ \lam{\cal L}_{c}^{\theta}(x_{<t}^{k},u_t,x_{>t}^{k-1};p^{k-1})+\frac{1}{2}\|u_t-x_{t}^{k-1}\|^{2}\right\} \label{eq:inexact_primal}
\end{equation}
are available. For example, one could consider applying an accelerated composite gradient (ACG) method to the problem associated with \eqref{eq:inexact_primal} so that $x_{t}^{k}$ satisfies
\[
\exists r_{t}^{k} \quad \text{ s.t. } \quad
\begin{cases}
r_{k}^{t}\in\partial\left(\lambda{\cal L}_{c}^{\theta}(x_{<t}^{k},\cdot,x_{>t}^{k-1};p^{k-1})+\frac{1}{2}\|\cdot-x_{t}^{k-1}\|^{2}\right)(x_{t}^{k}),\\
\|r_{t}^{k}\|^2\leq\sigma^{2}\|x_{t}^{k-1}-x^{k}\|^{2},
\end{cases}
\]
for some $\sigma\in(0,1)$.

\appendix

\section{Proof of Lemma~\ref{lem:key_iter_relations} and Lemma~\ref{lem:Lagr_bds}(a)--(b)}
\label{app:tech_ineq}

Before giving the proofs, we present some auxiliary results. To avoid repetition, we assume the reader is already familiar with \eqref{eq:global_admm_consts}--\eqref{not:Delta_var}.

The proof of the first result can be found in \cite[Lemma B.2]{kong2021aidal}. 
\begin{lem}
\label{lem:tech_ADeltaZ_DeltaP}For any $(\zeta,\theta)\in[0,1]^{2}$
satisfying $\zeta\leq\theta^{2}$ and $(a,b)\in\rn\times\rn$, we have
that 
\begin{equation}
\|a-(1-\theta)b\|^{2}-\zeta\|a\|^{2}\geq\left[\frac{(1-\zeta)-(1-\theta)^{2}}{2}\right]\left(\|a\|^{2}-\|b\|^{2}\right).\label{eq:inexact_tech_DeltaP1}
\end{equation}
\end{lem}
The next result establishes some general bounds given by the updates
in \eqref{eq:x_update}. 
\begin{lem}
\label{lem:gen_prox_ineq}For every $i\geq 1$, index $t=1,\ldots,B$,
and $u_t \in {\cal H}_{t}$, it holds that 
\begin{align*}
 & \lam\left[{\cal L}_{c}^{\theta}(x_{<t}^{i},u_t,x_{>t}^{i-1};p^{i-1})-{\cal L}_{c}^{\theta}(x_{< t}^i, x_{t}^{i},x_{>t}^{i-1};p^{i-1})\right]+\frac{1}{2}\|u_t-x_{t}^{i-1}\|^{2}\\
 & \geq\frac{1}{2}\|\Delta x_{t}^{i}\|^{2}+\left(\frac{1-\lam m_{t}}{2}\right)\|u_t-x_{t}^{i}\|^{2}+\frac{\lam c}{2}\|A_{t}(u_t-x_{t}^{i})\|^{2}.
\end{align*}
\end{lem}

\begin{proof}
Let $i\geq 1$, $t=1,\ldots, B$, and $u_t\in {\cal H}_{t}$ be fixed,
and define $\mu:=1-\lam m_{t}$ and $\|\cdot\|_{\alpha}^{2}:=\left\langle \cdot,(\mu I+\lam cA_{t}^{*}A_{t})(\cdot)\right\rangle $.
Since the prox stepsize $\lam$ is chosen 
 in $(0,1/(2m)]$ and $m \ge m_t$ in view of \eqref{eq:pp_indep_defs},
 it follows that $\mu \ge 1/2$.
Using the optimality of $x_{t}^{i}$, assumption (A4), and the fact that $\lam{\cal L}_{c}^{\theta}(x_{<t}^{i},\cdot,x_{>t}^{i-1};p^{i-1})+\|\cdot-x_{t}^{i-1}\|^{2}/2$
is $1$-strongly convex with respect to $\|\cdot\|_{\alpha}^{2}$,
it follows that 
\begin{align*}
 & \lam{\cal L}_{c}^{\theta}(x_{< t}^i,x_{t}^{i},x_{>t}^{i-1};p^{i-1})+\frac{1}{2}\|\Delta x_{t}^{i}\|^{2}\\
 & \leq\lam{\cal L}_{c}^{\theta}(x_{<t}^{i},u_t,x_{>t}^{i-1};p^{i-1})+\frac{1}{2}\|u_t-x_{t}^{i-1}\|^{2}-\frac{1}{2}\|u_t-x_{t}^{i}\|_{\alpha}^{2}\\
 & =\lam{\cal L}_{c}^{\theta}(x_{<t}^{i},u_t,x_{>t}^{i-1};p^{i-1})+\frac{1}{2}\|u_t-x_{t}^{i-1}\|^{2}-\frac{\mu}{2}\|u_t-x_{t}^{i}\|^{2}-\frac{\lam c}{2}\|A_{t}(u_t-x_{t}^{i})\|^{2}.
\end{align*}

\end{proof}

We are now ready to give the proof of Lemma~\ref{lem:key_iter_relations}.

\begin{proof}[Proof of Lemma~\ref{lem:key_iter_relations}]
(a) Using the definition of ${\cal L}_{c}^{\theta}(\cdot;\cdot)$ in \eqref{eq:dal}
and the relation in Lemma~\ref{lem:refine_props}(a), we conclude that
\begin{align}
{\cal L}_{c}^{\theta}(x^{i};p^{i})-{\cal L}_{c}^{\theta}(x^{i};p^{i-1}) & =(1-\theta)\left\langle \Delta p^{i},f^{i}\right\rangle =\left(\frac{1-\theta}{\chi c}\right)\|\Delta p^{i}\|^{2}+\frac{a_{\theta}}{\chi c}\left\langle \Delta p^{i},p^{i-1}\right\rangle \nonumber \\
 & =\left(\frac{1-\theta}{\chi c}\right)\|\Delta p^{i}\|^{2}+\frac{a_{\theta}}{\chi c}\left(\left\langle p^{i},p^{i-1}\right\rangle -\|p^{i-1}\|^{2}\right)  \nonumber\\
 & =\left(\frac{1-\theta}{\chi c}\right)\|\Delta p^{i}\|^{2}+\frac{a_{\theta}}{\chi c}\left(\frac{1}{2}\|p^{i}\|^{2}-\frac{1}{2}\|\Delta p^{i}\|^{2}-\frac{1}{2}\|p^{i-1}\|^{2}\right)  \nonumber\\
 & =\frac{b_{\theta}}{2\chi c}\|\Delta p^{i}\|^{2}+\frac{a_{\theta}}{2\chi c}\left(\|p^{i}\|^{2}-\|p^{i-1}\|^{2}\right). \label{eq:Lagr_p_incr}
\end{align}

(b) Using the definition of $m$ in \eqref{eq:pp_indep_defs} and summing the inequality of Lemma~\ref{lem:gen_prox_ineq} with $u_t=x_{t}^{i-1}$ from $t=1$ to $B$, we have that
\begin{align*}
\left(1-\frac{\lam m}{2}\right) \|\Delta x^{i}\|^{2}+\frac{\lam c}{2}\sum_{t=1}^{ B}\|A_{t}\Delta x_{t}^{i}\|^{2} & \leq\sum_{i=1}^{t}\left(1-\frac{\lam m_{t}}{2}\right)\|\Delta x_{t}^{i}\|^{2}+\frac{\lam c}{2}\sum_{t=1}^{ B}\|A_{t}\Delta x_{t}^{i}\|^{2}\\
 & \leq\lam\left[{\cal L}_{c}^{\theta}(x^{i-1};p^{i-1})-{\cal L}_{c}^{\theta}(x^{i};p^{i-1})\right].
\end{align*}
The conclusion now follows from dividing the above inequality by $\lambda$ and using the fact that $\lam \leq 1/m$.

(c) Note that the definition of $b_\theta$ in \eqref{eq:global_admm_consts} and \eqref{eq:chi_theta_cond} imply 
\[
\zeta := 2B\chi b_\theta \leq \theta^2.
\]
Hence, using the definition of $\gamma_\theta$ in \eqref{eq:global_admm_consts}, and Lemma~\ref{lem:tech_ADeltaZ_DeltaP}
with $(a,b)=(\Delta p^{i},\Delta p^{i-1})$
it follows that
\begin{equation}
\|\Delta p^{i}-(1-\theta)\Delta p^{i-1}\|^{2}\geq 2B\chi b_{\theta}\|\Delta p^{i}\|^{2}+\chi\gamma_{\theta}\left(\|\Delta p^{i}\|^{2}-\|\Delta p^{i-1}\|^{2}\right).\label{eq:DDeltap_bd}
\end{equation}
Using \eqref{eq:DDeltap_bd} at $i$ and $i-1$, Lemma~\ref{lem:refine_props}(a), and
the relation $\|a\|_1^2 \leq n\|a\|_2^2$ for $a\in\rn$, we have that 
\begin{align*}
\frac{c}{4}\sum_{t=1}^{ B}\|A_{t}\Delta x_{t}^{i}\|^{2} & \geq\frac{c}{4 B}\|A\Delta x^{i}\|^{2}=\frac{\|\Delta p^{i}-(1-\theta)\Delta p^{i-1}\|^{2}}{4 B\chi^{2}c}\\
 & \geq\frac{1}{4 B\chi c}\left[2 B b_{\theta}\|\Delta p^{i}\|^{2}+\gamma_{\theta}\left(\|\Delta p^{i}\|^{2}-\|\Delta p^{i-1}\|^{2}\right)\right]\\
 & =\frac{b_{\theta}}{2\chi c}\|\Delta p^{i}\|^{2}+\frac{\gamma_{\theta}}{4 B\chi c}\left(\|\Delta p^{i}\|^{2}-\|\Delta p^{i-1}\|^{2}\right). 
\end{align*}
\end{proof}
Next, we give the proof of Lemma~\ref{lem:Lagr_bds}(a)--(b).
\begin{proof}[Proof of Lemma~\ref{lem:Lagr_bds}(a)--(b)]
(a) 
Using Lemma~\ref{lem:key_iter_relations}(a), the definition of ${\cal L}_c^\theta(\cdot;\cdot)$ in \eqref{eq:dal}, the fact that $\theta\in(0,1)$, and the relations 
$2\left<a,b\right> \leq \|a\|^2 + \|b\|^2$ and $\|a+b\|^2 \leq 2\|a\|^2 + 2\|b\|^2$ for $a,b\in\rn$, it follows that
\begin{align*}
{\cal L}_{c}^{\theta}(x^{j};p^{j}) & =\phi(x^{j})+(1-\theta)\left\langle p^{i},f^{i}\right\rangle +\frac{c}{2}\|f^{i}\|^{2}\\
 & \overset{\text{L.\ref{lem:key_iter_relations}(a)}}{=} \frac{(1-\theta)}{\chi c}\left\langle p^{i},p^{i}-(1-\theta)p^{i-1}\right\rangle +\frac{1}{2c\chi^{2}}\|p^{i}-(1-\theta)p^{i-1}\|^{2}\\
 & \leq\frac{(1-\theta)}{2\chi c}\|p^{i}\|^{2}+\frac{(1-\theta)}{2\chi c}\|p^{i}-(1-\theta)p^{i-1}\|^{2}+\frac{1}{2\chi^{2}c}\|p^{i}-(1-\theta)p^{i-1}\|^{2}\\
 & \leq\frac{1}{2\chi c}\|p^{i}\|^{2}+\frac{1}{\chi^{2}c}\|p^{i}-(1-\theta)p^{i-1}\|^{2}\\
 & \leq\frac{1}{2\chi c}\|p^{i}\|^{2}+\frac{2}{\chi^{2}c}\|p^{i}\|^{2}+\frac{2}{\chi^{2}c}\|p^{i-1}\|^{2}\leq\frac{3(\|p^{i}\|^{2}+\|p^{i-1}\|^{2})}{\chi^{2}c}.
\end{align*}

(b) It holds that 
\begin{align*}
{\cal L}_{c}^{\theta}(x^{k};p^{k}) &=\phi(x^{k})+(1-\theta)\left\langle p^{k},f^k\right\rangle +\frac{c}{2}\|f^k\|^{2}\\
 & =\phi(x^{k})+\frac{1}{2}\left\Vert\frac{(1-\theta)p^{k}}{\sqrt{c}}+\sqrt{c} f^k\right\Vert ^{2}-\frac{(1-\theta)^{2}\|p^{k}\|^{2}}{2c} \\
 & \geq \phi(x^{k})-\frac{(1-\theta)^2\|p^{k}\|^{2}}{2c} \geq \phi(x^{k})-\frac{\|p^{k}\|^{2}}{2c}.
\end{align*}
\end{proof}

\bibliographystyle{siamplain}
\bibliography{relaxed_admm_ref}

\end{document}